\documentclass{amsart}

\usepackage{graphicx}\usepackage{color}
\usepackage{amsmath, amsthm, amssymb}
\newtheorem{thm}{Theorem}[section]
\newtheorem{cor}[thm]{Corollary}
\newtheorem{lem}[thm]{Lemma}
\newtheorem{conj}[thm]{Conjecture}
\newtheorem{rmk}[thm]{Remark}
\newtheorem{lemma}[thm]{Lemma}
\newtheorem{definition}[thm]{Definition}
\newtheorem{example}[thm]{Example}
\usepackage{enumerate}
\newcommand{\afk}{\widetilde{FK}_n}
\newcommand{\atheta}{\tilde{\theta}}
\newcommand{\fk}{{FK_n}}
\newcommand{\ac}{\mathbb{A}_{\mathcal{C}}}
\newcommand{\at}{\mathbb{A}_{\mathcal{C}_{Fl_n}}}

\newcommand{\h}{{\bf h}}
\newcommand{\p}{{\bf p}}
\newcommand{\x}{{\bf x}}
\newcommand{\y}{{\bf y}}

\newcommand{\af}{\text{af}}
\newcommand{\s}{{\bf s}^{(k)}}
\newcommand{\f}{\tilde{F}}
\newcommand{\sn}{\widetilde{S}_n}
\newcommand{\fl}{\hat{Fl}}
\newcommand{\e}{\mathcal{E}}
\newcommand{\gr}{\hat{Gr}}
\newcommand{\asp}{\tilde{\mathfrak{S}}}

\newcommand{\homa}{\Hom_\mathbb{Q}(\mathbb{A},\mathbb{A})}
\newcommand{\homq}{\Hom_\mathbb{Q}(\mathbb{A},\mathbb{Q})}

\DeclareMathOperator{\supp}{\text{Supp}}
\DeclareMathOperator{\comp}{\text{ascomp}}
\DeclareMathOperator{\Hom}{\text{Hom}}

\renewenvironment{proof}{{\bf \emph {Proof.}}}{\qed}

\title[\resizebox{4.5in}{!}{Combinatorial description of the cohomology of the affine flag variety}]{Combinatorial description of the cohomology of the affine flag variety}
\author{Seung Jin Lee}
\address{Department of Mathematical Sciences, Seoul National University, GwanAkRo 1, Gwanak-Gu, Seoul 08826, Korea}      
\email{lsjin@snu.ac.kr}
\date{}

\begin{document}
\maketitle
\begin{abstract}
We define a polynomial representative of the Schubert class in the cohomology of an affine flag variety associated to $SL(n)$, called an affine Schubert polynomial. Affine Schubert polynomials are defined by using divided difference operators, generalizing those operators used to define Schubert polynomials, so that Schubert polynomials are special cases of affine Schubert polynomials. Also, affine Stanley symmetric functions can be obtained from affine Schubert polynomials by setting certain variables to zero. We study affine Schubert polynomials and divided difference operators by constructing an affine analogue of Fomin-Killilov algebra called an affine Fomin-Kirillov algebra. We introduce Murnaghan-Nakayama elements and Dunkl elements in the affine Fomin-Kirillov algebra to describe the cohomology of the affine flag variety and affine Schubert polynomials, and by doing so we also obtain a Murnaghan-Nakayama rule for the affine Schubert polynomials. 
\end{abstract}

\section{Introduction}
Schubert calculus is a branch of algebraic geometry introduced by Hermann Schubert, in order to solve various counting problems of projective geometry. Schubert studied geometric objects, now called Schubert varieties, which are certain closed varieties in a Grassmannian. The intersection theory of these varieties can be studied via the product structure of the cohomology of the Grassmannian by the solution of Hilbert's Fifteenth problem. Schubert varieties give Schubert classes, elements in the Schubert basis of the cohomology of the Grassmannian, and the Schur functions are polynomial representatives of the Schubert classes. Therefore, Schur functions play an important role in understanding Schubert calculus. Generalizing this phenomena, many mathematicians have defined and studied polynomial representatives for Schubert classes in various cohomology theories of varieties such as Grassmannians and flag varieties.\\

Affine Schubert calculus is a generalization of Schubert calculus for the affine Grassmannian and the affine flag variety instead of the Grassmannian and the flag variety. The (equivariant) homology of the affine flag variety and the affine Grassmannian can be identified with the nilHecke ring and the Peterson subalgebra by the work of Kostant, Kumar, and Peterson \cite{KK86,Pet97} providing an algebraic framework to understand the homology. The combinatorial theory of the affine Schubert calculus was facilitated and studied extensively since Lam \cite{Lam08} showed that the affine Schur functions and the $k$-Schur functions are polynomial representatives of the Schubert classes in the cohomology and the homology of the affine Grassmannian associated to $SL(n)$. Note that $k$-Schur functions were introduced by Lapointe, Lascoux, and Morse \cite{LLM03} during the study of the Macdonald positivity conjecture, suggesting that a full development of the combinatorics in affine Schubert calculus will be substantial to understand the Macdonald theory as well. For more details about the affine Schubert calculus, see \cite{LLMSSZ}.\\

Compared to the combinatorics of the affine Grassmannian, the situation for the affine flag variety is much less known. The best way (to the author's best knowledge) to describe the cohomology of the affine flag variety in the past literature is by taking the dual of the nilHecke algebra \cite{KK86} which makes the combinatorics rather mysterious. There are several serious obstacles to understanding the cohomology of the affine flag variety. The fact that Borel's characteristic map is not surjective for the affine flag variety implies that its cohomology is not generated by degree one elements, resulting in the absence of the affine Schubert polynomials and their combinatorics in terms of degree one elements. In this paper, we study higher degree elements that are not generated by degree one elements and use them to define the affine Schubert polynomials, the polynomial representatives of Schubert classes in the cohomology of the affine flag variety associated to $SL(n)$. \\

Before we define the affine Schubert polynomials, we first need to describe the cohomology of the affine flag variety in terms of a polynomial ring. Let $G$ be $SL(n)$, $B$ Borel subgroup, $T$ the maximal torus, $K$ the maximal compact form of $G$, $T_{\mathbb{R}}=K\cap T$, $\Omega K$ the topological group of based loops into $K$, and $LK$ the space of loops into $K$. We call $\Omega K$ an affine Grassmannian associated to $SL(n)$, denoted by $\gr$, and $LK/T_{\mathbb{R}}$ an affine flag variety associated to $G$, denoted by $\fl$. By the K\"unneth isomorphism and the isomorphism $LK \cong \Omega K \times K$, we have an isomorphism 
\begin{align}\label{quillen}
H^*(LK/T_{\mathbb{R}})\cong H^*(\Omega K) \otimes_\mathbb{Q} H^* (K/T_{\mathbb{R}})
\end{align}
as $\mathbb{Q}$-algebras where the coefficient ring is over $\mathbb{Q}$. The cohomology rings on the right can be described as follows. Since $K/T_{\mathbb{R}}$ is the flag variety, the cohomology of $K/T_{\mathbb{R}}$ is isomorphic to $\mathbb{Q}[x_1, \ldots,x_n]/\{e_j(x_1,\ldots,x_n)\}$, denoted by $R_{Fl_n}$, where $e_j(x_1,\ldots,x_n)$ is the elementary symmetric polynomial of degree $j$ in $x_1,\ldots,x_n$. For an element $w$ in the symmetric group $S_n$,  Lascoux and Sch\"{u}tzenberger \cite{LS82} defined the Schubert polynomial $\mathfrak{S}_w$ that maps to the Schubert class $\xi^w_{Fl_n}$ in $H^*(K/T_{\mathbb{R}})$. On the other hand, $H^*(\gr)$ is isomorphic to $\mathbb{Q}[p_1,\ldots,p_{n-1}]$ where $p_i$ is the power sum symmetric function of degree $i$ in another set of variables $y_1,y_2,\ldots$, and Lam \cite{Lam08} showed that the Schubert class $\xi^w_{\gr}$ in $H^*(K/T_{\mathbb{R}})$ for an element $w$ in $\sn/S_n$ maps to the affine Stanley symmetric function, where $\sn$ denote the affine symmetric group. We denote $\mathbb{Q}[p_1,\ldots,p_{n-1}]$ by $R_{\gr}$.\\

Let $R_n$ be the tensor product $R_{\gr}\otimes_\mathbb{Q} R_{Fl_n}$ which is isomorphic to $H^*(LK/T_{\mathbb{R}})$. The cohomology of the affine flag variety $H^*(LK/T_{\mathbb{R}})$ also has Schubert basis $\xi^w$ indexed by an element in $\sn$ (see Section \ref{affine}). One of main results in this paper is a description of the affine Schubert polynomial $\asp_w$, the element of $R_n$ corresponding to the Schubert basis element $\xi^w$, in terms of divided difference operators. We first define divided difference operator $\partial_i$ for $i\in \mathbb{Z}/n\mathbb{Z}$ acting on $R_n$.

\begin{definition}\label{bgg}
For $i\in \mathbb{Z}/n\mathbb{Z}$, the Weyl group action $s_i$ and the divided difference operator $\partial_i:={1-s_i \over x_i-x_{i+1}}$ on $R_n$ can be uniquely defined by the following rules.
\begin{enumerate}
\item For $f,g \in H^*(\fl)$, we have $s_i(fg)=s_i(f)s_i(g)$. Therefore $\partial_i$ satisfies the Leibniz's rule: for $f,g \in H^*(\fl)$, we have
$$\partial_i(fg)= \partial_i(f)g + s_i(f) \partial_i(g).$$
\item For nonzero $i$ and for all $m$, we have 
\begin{align*}
s_i(p_m)&=p_m \\
\partial_i(p_m)&=0.
\end{align*}
\item For $i=0$, we have
\begin{align*}
s_0(p_m)&= p_m+x_1^m-x_0^m\\
\partial_0(p_m)&= \sum_{j=0}^{m-1} x_1^{m-1-j}x_0^j
\end{align*}

\item For all $i,j\in \mathbb{Z}/n\mathbb{Z}$, we have 
\begin{align*}
s_i(x_j)&=x_{s_i(j)} \\
\partial_i(x_j)&=\delta_{ij}-\delta_{i,j+1}.
\end{align*}

\end{enumerate}
\end{definition}

The ring $R_n$ is graded under the following degree: $\deg(x_i)=1$ and $\deg(p_j)=j$.
Now we are ready to state the main theorem, existence of affine Schubert polynomials.
\begin{thm}\label{asp}
For $w \in \sn$, there exists a unique homogeneous element $\asp_w$ of degree $\ell(w)$ in $R_n$ satisfying
$$\partial_i\asp_w=\begin{cases} \asp_{ws_i} & \text{ if } \ell(ws_i)=\ell(w)-1 \\ 0 & \text{ otherwise}.\end{cases}$$
for $i\in \mathbb{Z}/n\mathbb{Z}$, with the initial condition $\asp_{id}=1$. Moreover, under the isomorphism $ H^*(LK/T_{\mathbb{R}}) \cong R_n$, a Schubert class $\xi^w$ sends to $\asp_w$. We call $\asp_w$ an affine Schubert polynomial for $w$.
\end{thm}

Since affine Schubert polynomials represent Schubert basis $\xi^w$, they are naturally related with Schubert polynomials and affine Stanley symmetric functions. Let $pr_1$ be the inclusion $\Omega K\rightarrow LK/T_{\mathbb{R}}$ and $pr_2$ be the inclusion $K/T_{\mathbb{R}}\rightarrow LK/T_{\mathbb{R}}$. Then the pullback $pr_1^*: R_n \cong H^*(LK/T_{\mathbb{R}})\rightarrow H^*(\Omega K) \cong R_{\gr}$ is the map sending all $x_i$'s to zero and the pullback $pr_2^*: R_n \cong H^*(LK/T_{\mathbb{R}})\rightarrow H^*(K/T_{\mathbb{R}} )\cong R_{Fl_n}$ is the map sending all $p_j$'s to zero. Under this setup, we have

\begin{thm} \label{restriction}
For $w$ in $\sn$, the pull-back image $pr_1^*(\asp_w)$ is the same as the affine Stanley symmetric function $F_w$. Also, the image $pr_2^*(\asp_w)$ is the same as the Schubert polynomial $\mathfrak{S}_w$ if $w \in S_n$, and zero otherwise.
\end{thm}

We also provide Murnaghan-Nakayama rule, the expansion of $p_m \asp_w$ in terms of affine Schubert polynomials (See Section 9). To prove the existence and properties of the affine Schubert polynomials, we first interpret the cohomologies of the affine flag variety, the affine Grassmannian and the flag variety as subalgebras of $\Hom_{\mathbb{Q}}(\mathbb{A},\mathbb{A})$ where $\mathbb{A}$ is the affine nilCoxeter algebra studied by Kostant and Kumar \cite{KK86}. Then we define an affine analogue of the Fomin-Kirillov algebra \cite{FK99}, called an \emph{affine Fomin-Kirillov algebra}, to study combinatorics of these subalgebras. We define a map from the affine Fomin-Kirillov algebra $\afk$ to $\homa$ defined by Bruhat action, then we show that the subalgebra generated by Dunkl elements and Murnaghan elements in $\afk$ maps to the subalgebra of $\homa$ isomorphic to the cohomology of the affine flag variety, and the subalgebra generated by Dunkl elements (resp. Murnaghan elements) in $\afk$ maps to the subalgebra of $\homa$ isomorphic to the cohomology of the flag variety (resp. the affine Grassmannian). Dunkl elements in $\afk$ is an affine analogue of Dunkl elements in Fomin-Kirillov algebra \cite{FK99}, but the definition of Murnaghan elements is new and the definition naturally contains Murnaghan Nakayama rule of the affine Schubert polynomials. We also generalize the Fomin-Kirillov conjecture in the affine Fomin-Kirillov algebra (Section 8). Note that a lot of variations for the Fomin-Kirillov algebra have been studied \cite{Kir15,KM03,KM04,KM05,KM10,KM12}. For example, there are generalizations of the Fomin-Kirillov algebra for K-theory, quantum, equivariant cohomology and for other finite types, so developing their affine analogue would be an interesting problem. Note that the isomorphism (\ref{quillen}) can be generalized for other types and for equivariant cohomologies. \\

There are some direct applications of Murnaghan-Nakayama rule for the affine Schubert polynomials. By Theorem \ref{restriction}, we also obtain Murnaghan-Nakayama rule for affine Stanley symmetric functions and by a duality, we write $k$-Schur functions in terms of power sum symmetric functions using $k$-strong-ribbon tableaux. This expansion gives the character table of the representation of the symmetric group whose Frobenius characteristic is the $k$-Schur function. Bandlow, Schilling, Zabrocki studied $k$-weak-ribbons, called $k$-ribbons in \cite{BSZ10}, to describe the MN rule for the $k$-Schur functions. $k$-strong-ribbons are combinatorial objects dual to $k$-weak-ribbons, since $k$-weak-ribbons give the MN rule for $k$-Schur functions and $k$-strong-ribbons give a MN rule for affine Stanley symmetric functions.\\

The paper is structured as follows: In Section 2, we review affine symmetric groups, symmetric functions, affine flag varieties, affine Grassmannians, affine nilCoxeter algebra, and the Fomin-Kirillov algebra. In Section 3, we define the affine Fomin-Kirillov algebra (the affine FK algebra in short) and study the Bruhat operator acting on the affine nilCoxeter algebra. In Section 4, we define Dunkl elements and Murnaghan-Nakayama elements, and investigate relations among those elements. In Section 5, we derive identities that uniquely determine Bruhat operators for Murnaghan-Nakayama elements. In Section 6, we recall properties of cap operators and BSS operators and show an equality between Bruhat operators for MN elements and certain cap operators (Theorem \ref{main}). In Section 6, we define the divided difference operators acting on the affine FK algebra. In Section 7, we define and discuss the divided difference operators acting on $R_n$ and the affine Schubert polynomials. In Section 8, we state the nonnegativity conjecture in the affine FK algebra. In section 9, we state the Murnaghan-Nakayama rules both for the affine Schubert polynomial and the affine Stanley symmetric functions. In Section 10, we discuss a new formula for $k$-Schur functions in terms of power sum symmetric functions as well as its relation with representation theory.

\section{Preliminaries}
\subsection{Affine symmetric group}
Let $I$ be the set $\{0,1,\ldots,n-1\}=\mathbb{Z}/n\mathbb{Z}$. Let $\sn$ denote the affine symmetric group with simple generators $s_0,s_1,\ldots,s_{n-1}$ satisfying the relations
\begin{align*}
s_i^2&=1&\\
s_is_{i+1}s_i&=s_{i+1}s_is_{i+1}\\
s_is_j&=s_js_i&& \text{if } i-j\neq 1,-1,
\end{align*} 
where indices are taken modulo $n$. An element of the affine symmetric group may be written as a word in the generators $s_i$. A \emph{reduced word} of the element is a word of minimal length. The \emph{length} of $w$, denoted $\ell(w)$, is the number of generators in any reduced word of $w$. \\

The \emph{Bruhat order}, also called \emph{strong order}, on affine symmetric group elements is a partial order where $u<w$ if there is a reduced word for $u$ that is a subword of a reduced word for $w$. If $u<w$ and $\ell(u)=\ell(w)-1$, we write $u\lessdot w$. It is well-known that $u\lessdot w$ if and only if there exists a transposition $t_{ab}$ in $\sn$ such that $w=ut_{ab}$ and $\ell(u)=\ell(w)-1$. See \cite{BB05} for instance.\\

The subgroup of $\sn$ generated by $\{s_1,\cdots,s_{n-1}\}$ is naturally isomorphic to the symmetric group $S_n$. The $0$-Grassmannian elements are minimal length coset representatives of $\sn/ S_n$. In other words, $w$ is $0$-Grassmannian if and only if all reduced words of $w$ end with $s_0$. More generally, for $i \in \mathbb{Z}/n\mathbb{Z}$ and $w \in \sn$, $w$ is called $i$-Grassmannian if all reduced words of $w$ end with $s_i$. We denote the set of $i$-Grassmannian elements by $\sn^i$. \\

For $a\in \mathbb{Z}$ and $u,w \in \sn$, a \emph{marked strong cover} $(u\overset{(j_1,j_2)}\longrightarrow w)$ with respect to $a$ is defined by $w\lessdot wt_{j_1,j_2}=u$ with $j_1\leq a< j_2$ and $\ell(w)=\ell(u)-1$. For this cover, we distinguish two strong covers corresponding to $t_{j_1,j_2}$ and $t_{j_1+n,j_2+n}$ so that we may have multiple covers between fixed $w,u$. Note that $w(j_2)$ is called the marking of the strong cover in \cite{BSS13,BSS14,LLMS10} and the marking is used to define strong strips and $k$-Schur functions. 
  
\subsection{Symmetric functions}
Let $\Lambda$ denote the ring of symmetric functions over $\mathbb{Q}$. For a partition $\lambda$, we let $m_\lambda,h_\lambda,p_\lambda,s_\lambda$ denote the monomial, homogeneous, power sum and Schur symmetric functions, respectively, indexed by $\lambda$. Each of these families forms a basis of $\Lambda$. Let $\langle\cdot,\cdot\rangle$ be the Hall inner product on $\Lambda$ satisfying $\langle m_\lambda,h_\mu\rangle=\langle s_\lambda,s_\mu \rangle=\delta_{\lambda,\mu}$ for partitions $\lambda,\mu$. See \cite{Sta99} for details of symmetric functions. \\

Let $\Lambda_{(k)}$ denote the subalgebra generated by $h_1,h_2,\ldots,h_k$. The elements $h_\lambda$ with $\lambda_1\leq k$ form a basis of $\Lambda_{(k)}$. We call a partition $\lambda$ with $\lambda_1\leq k$ a \emph{$k$-bounded partition}. Note that there is a bijection between the set of $k$-bounded partitions and the set of $0$-Grassmannian elements in $\sn$ \cite{LM05}. We denote by $\lambda(w)$ the $k$-bounded partition corresponding to the $0$-Grassmannian element $w$. Let $\Lambda^{(k)}=\Lambda/I_k$ denote the quotient of $\Lambda$ by the ideal $I_k$ generated by $m_\lambda$ with $\lambda_1>k$. The image of the elements $m_\lambda$ with $\lambda_1\leq k$ form a basis of $\Lambda^{(k)}$. Note that $I_k$ is isomorphic to the ideal generated by $p_\lambda$ for $\lambda_1>k$, so that $p_\lambda$ for $k$-bounded partitions $\lambda$ form a basis of $\Lambda^{(k)}$.\\

There is another remarkable basis for $\Lambda_{(k)}$ and $\Lambda^{(k)}$. For a $k$-bounded partition $\lambda$, a $k$-Schur function $s^{(k)}_\lambda$  and an affine Schur function $\f_\lambda$ are defined in \cite{Lam06,LLMS10}. Lam \cite{Lam08} showed that the $k$-Schur functions (resp. the affine Schur functions) are representatives of the Schubert basis of the homology (resp. the cohomology) of the affine Grassmannian $\gr$ for $SL(n)$ via the isomorphisms of Hopf-algebras
\begin{align*}
\Lambda_{(k)}&\cong H_*(\gr) \\
\Lambda^{(k)}&\cong H^*(\gr).
\end{align*}
The restriction of the Hall inner product on $\Lambda^{(k)}\times \Lambda_{(k)}$ gives the identity $\langle \f_\lambda,s^{(k)}_\mu\rangle=\delta_{\lambda,\mu}$.
Since we do not use the definitions of $k$-Schur functions, affine Schur functions, and affine Stanley symmetric functions in this paper, definitions are omitted. See \cite{Lam06,Lam08,LLMS10} for more details.

\subsection{Affine flag varieties and affine Grassmannians} \label{affine}
We define the Kac-Moody flag variety $\fl$, the affine Grassmannian and the Schubert basis on the (co)homology of the affine flag variety and the affine Grassmannian in algebraic way and compare with the definitions of these defined in the introduction. There are two (algebraic) definitions of the Kac-Moody flag variety $\fl$ in \cite{LLMSSZ}, but we only recall $\fl$ as the Kac-Moody flag ind-variety in \cite{KK86,Kum02}. \\

Let $G_\af$ denote the Kac-Moody group of affine type associated with $G:=SL(n)$ and let $B_\af$ denote its Borel subgroup. The Kac-Moody flag ind-variety $\fl=G_\af/B_\af$ is paved by cells $B_\af\dot{w}B_\af/B_\af\cong\mathbb{C}^{\ell(w)}$ whose closure $X_w$ is called the Schubert variety where $\dot{w}$ is a representative element in $G_\af$ corresponding to $w$ in $\sn$. A Schubert variety defines a Schubert class $\xi^w \in H^*(\fl)$ and $\xi_w \in H_*(\fl)$. The affine Grassmannian $\gr$ is $G_\af/P_\af$ where $P_\af$ is the maximal parabolic subgroup obtained by ``omitting the zero node". \\

Quillen \cite{Qui} prove that there are weak homotopy equivalences $\Omega K \cong \gr$ and $LK/T_{\mathbb{R}} \cong \fl$ under which the actions of $T_{\mathbb{R}}$ and $T$ correspond, so that we can identify these spaces if we are working on (co)homologies of these. Lam \cite{Lam08} showed that the affine Stanley symmetric function $\tilde{F}_w$ is the pullback $pr_1^*(\xi^w)$ in $H^*(\gr)\cong \Lambda^{(k)}$ where $pr_1:\Omega K\rightarrow LK/T_\mathbb{R}$ . Note that we also have a surjection $q_1: \fl \rightarrow \gr$ such that $q_1\circ pr_1$ is the identity up to homotopy, and an evaluation map $ev: \fl \cong LK/T_\mathbb{R} \rightarrow K/T_\mathbb{R}\cong Fl_n$ sending a loop $f$ to its starting point $f(1_K)$.\\
By (\ref{quillen}), we have an isomorphism of $\mathbb{Q}$-algebras
\begin{align}\label{isom1}
H^*(\fl)\cong H^*(\gr) \otimes_\mathbb{Q} H^*(Fl_n),
\end{align}
and an isomorphism of $\mathbb{Q}$-modules
\begin{align}\label{isom2}
H_*(\fl)\cong H_*(\gr) \otimes_\mathbb{Q} H_*(Fl_n),
\end{align}
 where $Fl_n$ is the flag variety.

\subsection{Affine nilCoxeter algebra}\label{affinenilcoxeter}

In this subsection, we review the theory of the affine nilCoxeter algebra and its connection with the previous section.

The \emph{affine nilCoxeter algebra} $\mathbb{A}$ is the algebra generated by $A_0,A_1,\ldots,A_{n-1}$ over $\mathbb{Q}$, satisfying
\begin{align*}
A_i^2&=0&\\
A_iA_{i+1}A_i&=A_{i+1}A_i A_{i+1}\\
A_iA_j&=A_j A_i&& \text{if } i-j\neq 1,-1.
\end{align*} 
where the indices are taken modulo $n$. The subalgebra $\mathbb{A}_f$ of $\mathbb{A}$ generated by $A_i$ for $i\neq 0$ is isomorphic to the nilCoxeter algebra studied by Fomin and Stanley \cite{FS94}. The simple generators $A_i$ are considered as the \emph{divided difference operators} (see Section \ref{divided}). \\

The $A_i$ satisfy the same braid relations as the $s_i$ in $\sn$, i.e., $A_iA_{i+1}A_i=A_{i+1}A_iA_{i+1}$. Therefore it makes sense to define
\[ \begin{array} {rlll} A_w&=&A_{i_1}\cdots A_{i_l} & \mbox{where} \\ w&=&s_{i_1}\cdots s_{i_l} & \mbox{is a reduced decomposition.} \end{array} \]

One can check that 
\[A_vA_w=\left\{ \begin{array}{ll} A_{vw} & \mbox{if } \ell(vw)=\ell(v)+\ell(w) \\ 0 & \mbox{otherwise.}\end{array} \right. \]

There is a coproduct structure on $\mathbb{A}$ defined by
$$\Delta(A_w)=\sum p^w_{u,v}A_u\otimes A_v$$
where the sum is over all $u,v\in \sn$ satisfying $\ell(w)=\ell(u)+\ell(v)$. Kostant and Kumar \cite{KK86} showed that $p^w_{u,v}$ is the same as the structure coefficient for the cohomology of the affine flag variety. Note that $p^w_{u,v}$'s are nonnegative integers \cite{Gra01}.\\

A word $s_{i_1}s_{i_2}\cdots s_{i_l}$ with indices in $\mathbb{Z}/n\mathbb{Z}$ is called \emph{cyclically decreasing} if each letter occurs at most once and whenever $s_i$ and $s_{i+1}$ both occur in the word, $s_{i+1}$ precedes $s_i$. For $J\varsubsetneq\mathbb{Z}/n\mathbb{Z}$, a \emph{cyclically decreasing element} $w_J$ is the unique cyclically decreasing permutation in $\sn$ which uses exactly the simple generators in $\{s_j \mid j\in J\}$. For $i <n$, let
\[\h_i=\sum\limits_{\substack{ J\subset \mathbb{Z}/n\mathbb{Z} \\ |J|=i }} A_{w_J} \in \mathbb{A} \]
where $\h_0=1$ and $\h_i=0$ for $i< 0$ by convention. Lam \cite{Lam06} showed that the elements $\{\h_i\}_{i<n}$ commute and freely generate a subalgebra $\mathbb{B}$ of $\mathbb{A}$ called the \emph{affine Fomin-Stanley algebra}. It is well-known that $\mathbb{B}$ is isomorphic to $\Lambda_{(k)}$ via the map sending $\h_i$ to $h_i$. Therefore, the set $\{\h_\lambda=\h_{\lambda_1}\ldots \h_{\lambda_l} \mid \lambda_1\leq k\}$ forms a basis of $\mathbb{B}$. \\

There is another basis of $\mathbb{B}$, called the \emph{noncommutative $k$-Schur functions} $\s_\lambda$. For a bounded partition $\lambda$, the noncommutative $k$-Schur function $\s_\lambda$ is the image of $s^{(k)}_\lambda$ via the isomorphism $\Lambda_{(k)}\cong \mathbb{B}$. It is shown \cite{Lam08} that the noncommutative $k$-Schur function $\s_\lambda$ is the unique element in $\mathbb{B}$ that has the unique $0$-Grassmannian term $A_{w_\lambda}$ where $w_\lambda$ is the 0-Grassmannian element corresponding to $\lambda$. We also denote $s^{(k)}_\lambda$ by $s^{(k)}_{w(\lambda)}$. The noncommutative $k$-Schur functions are non-equivariant version of the $j$ functions studied by Peterson \cite{Pet97} for affine type $A$. For details and the original definition of noncommutative $k$-Schur functions, see \cite{Lam08,LLMSSZ}.\\

We recall the following theorem proved in \cite[Section 4.1]{BSS14}. 
\begin{thm}\label{basisa}
For an element $w\in \sn$, let $w=w^0 w^1$ be the unique decomposition with a 0-Grassmannian element $w^0$ and $w^1 \in S_n$. Then the set 
$$\{\s_{w^0}A_{w^1}\mid w \in \sn \}$$
forms a basis of $\mathbb{A}$.
\end{thm}
One can show Theorem \ref{basisa} by an induction on $\ell(w^0)$ with the fact that $\s_{w^0}$ has the unique $0$-Grassmannian term $A_{w^0}$.\\

In this paper, we identify $\mathbb{B}$ with $H_*(\gr)$ as Hopf-algebras over $\mathbb{Q}$ where the noncommutative $k$-Schur functions send to the Schubert homology classes. Also, we identify $\mathbb{A}$ (resp. $\mathbb{A}_f$) with $H_*(\fl)$ (resp. $H_*(Fl_n)$) as $\mathbb{Q}$-modules and identify $A_w$ with the Schubert basis $\xi_w$ in $H_*(\fl)$ (resp. in $H_*(Fl_n)$). The identification is natural since there is a coproduct structure on $\mathbb{A}$ such that coproduct structure constants are the same as the multiplicative structure of $H^*(\fl)$ in terms of Schubert basis.\\

Then we can consider a Schubert class $\xi^w$ (or any cohomology class) in $H^*(\fl)$ as an action on $\mathbb{A}$ defined by a cap product: for given $\xi^w$, the map sending a homology class $\xi_v$ to its cap product $\xi^w \cap \xi_v$ gives an algebra monomorphism $H^*(\fl)\rightarrow \homa$. We denote the image of $\xi^w$ by $D_w$ and we call a \emph{cap operator}. The author studied combinatorics of cap operators in \cite{Lee14} to prove Pieri rule for the affine flag variety. Note that there is a surjection from $\mathbb{A} \rightarrow \mathbb{Q}$ sending all $A_w$ to $1$ if $w=id$, and to $0$ otherwise. This induces a surjection from $\homa$ to $\homq$ and the composition $H^*(\fl) \rightarrow \homa \rightarrow \homq$ gives a natural pairing map

$$H^*(\fl) \times \mathbb{A} \cong H^*(\fl) \times H_*(\fl) \rightarrow \mathbb{Q}$$
such that $\xi^w$ (or $D_w$) and $\xi_v$ (or $A_v$) form a dual basis. In other words, the coefficient of $A_{\text{id}}$ in $D_w(A_v)$ is $\delta_{wv}$. In this paper, we often identify $H^*(\fl)$ with its image in $\homa$ or with $\homq$. Note that $\homq$ does not have a natural product structure, but we impose the product structure induced by those from $H^*(\fl)$.
\\

Since $\mathbb{B}$ (resp. $\mathbb{A}_f$) is identified with the homology of the affine Grassmannian (resp. flag variety), Theorem \ref{basisa} provides combinatorial interpretations of the isomorphism (\ref{isom1}) and (\ref{isom2}). The pullbacks and pushforwards of maps $pr_1,q_1,pr_2,ev$ can be purely written in terms of the affine nilCoxeter algebra. For example, the pushforward of the evaluation map is 
$$ev_*: \mathbb{A}\rightarrow \mathbb{A}_f$$ sending $\s_{w^0}A_{w^1}$ to 0 if $w^0$ is not the identity, and to $A_{w^1}$ otherwise. The evaluation map also induces the pullback $ev^*: H^*(Fl_n) \rightarrow H^*(\fl)$. In terms of the affine nilCoxeter algebra, the pullback is $ev^*: \Hom(\mathbb{A}_f,\mathbb{Q})\rightarrow \homq $ sending $g \in \Hom(\mathbb{A}_f,\mathbb{Q})$ to $ev^*(g)$ defined by $ev^*(g)(\s_{w^0}A_{w^1})=g(A_{w^1})$ if $w^0$ is the identity, and $=0$ otherwise. One can show $\s_{s_0}=\h_1=A_0+A_1+\ldots+A_{n-1}$ (see \cite{Lam06}),  and the image $ev^*(\xi_{Fl_n}^{s_i})$ for $i\neq 0$ is $\xi^{s_i}-\xi^{s_0}$ since $(\xi^{s_i}-\xi^{s_0})(\s_{s_0})=0$ and $(\xi^{s_i}-\xi^{s_0})(A_j)= \delta_{ij}$ for nonzero $j$. This computation shows the following theorem.
\begin{thm} \label{degree1image}
For $i=1,\ldots,n-1$, the image of $\xi^{s_i}_{Fl_n}$ under the pullback $ev^*: H^*(Fl_n) \rightarrow H^*(\fl)$ is $\xi^{s_i}-\xi^{s_0}$.

\end{thm}

\subsection{Fomin-Kirillov algebra}\label{finitefk}
We review some facts about the Fomin-Kirillov algebra proved in \cite{FK99}.
\begin{definition}
For a fixed positive integer $n$, let $\fk$ be the free algebra over $\mathbb{Q}$ generated by $\{[ij]: i,j\in \mathbb{Z},1\leq i<j \leq n\}$ with the following relations:
\begin{align*}
[ij]^2&=0.\\
[ij][kl]&=[kl][ij] \text{ for distinct } i, j, k ,l. \\
[ij][jk]&=[jk][ik]+[ik][ij] \text{ and } [jk][ij]=[ik][jk]+[ij][ik] \text{ for distinct } i, j, k.
\end{align*}
\end{definition}
Let $\theta_i$ be the Dunkl elements defined by $\sum_{j\neq i}[ij]$. Here for $i<j$, $[ji]$ is the same as $-[ij]$. Then the $\theta_i$'s commute pairwise for all $i$. Moreover, all symmetric functions in Dunkl elements vanish in $\fk$ and these are all relations between $\theta_i$. Therefore, the commutative subalgebra generated by Dunkl elements is isomorphic to the cohomology of the flag variety, proved by Fomin and Kirillov in \cite{FK99}. Under the isomorphism, a Dunkl element $\theta_i$ sends to $\xi_{Fl_n}^{s_{i+1}}-\xi_{Fl_n}^{s_{i}}$.

\section{Affine Fomin-Kirillov algebra}\label{affinefk}
For $i \in \mathbb{Z}$, let $\overline{i}$ be the residue of $i$ modulo $n$.
\begin{definition}\label{fk}
 Let $B$ be the free algebra over $\mathbb{Q}$ generated by $\mathcal{S}=\{[ij]: i,j\in \mathbb{Z},i<j,\overline{i}\neq\overline{j}\}$. Let $B(N)$ be the subalgebra of $B$ generated by elements $[ij]$ with $|i|,|j|\geq N$. Then we have a filtration
$$B=B(0)\supset B(1) \supset B(2) \supset \cdots.$$
Let $\mathcal{A}$ be the inverse limit $\varprojlim \left(B/B(i)\right)$.
\end{definition}

 An element $\x$ in $\mathcal{A}$ can be written as a (possibly infinite) sum $\sum_{J,m} a_{J,m}\x_{J,m}$ where $a_{J,m}\in \mathbb{Z}$, $J$ is in $\prod_{i=1}^m \mathcal{S}$ for $m\geq 0$, and $\x_{J,m}=[j_{1,1},j_{1,2}][j_{2,1},j_{2,2}]\cdots[j_{m,1},j_{m,2}]$ when $J=\left([j_{1,1},j_{1,2}],[j_{2,1},j_{2,2}],\cdots,[j_{m,1},j_{m,2}]\right)$. For $m=0$, $J$ is an empty set and we set $\x_{J,m}=1$. For $i>j$, we use the convention $[ij]=-[ji]$. We call $\x_{J,m}$ a \emph{noncommutative monomial} in $\mathcal{A}$.\\

Define the \emph{Bruhat action} of $[ij]$ on $\mathbb{A}$ by
\begin{equation}\label{bruhat}
A_w\cdot[ij] =
\begin{cases}
A_{wt_{ij}} & \text{if}\quad \ell(wt_{ij})=\ell(w)-1 \\
0 & \text{otherwise.}
\end{cases}
\end{equation}

For an element in $B$, one can define an action on $\mathbb{A}$ extended linearly. For an element $\x$ in $\mathcal{A}$, even if $\x$ is an infinite summation of product of $[ij]$'s, it is possible that all but finitely many terms $A_w \cdot \x_{J,m}$ vanish when acting by $\x$ on an element $A_w$. If this happens, we say that an element $\x$ gives a valid action on $\mathbb{A}$. Let $\mathcal{E}$ be the subalgebra of $\mathcal{A}$ consisting of elements which give a valid action on $\mathbb{A}$. Most elements in $\mathcal{A}$ in this paper are infinite sums but have a valid action on $\mathbb{A}$. Define the map $D: \mathcal{E} \rightarrow \homa$ by sending $\x$ to $D_\x$, where $D_\x(A_v):=A_v \cdot \x$. We call $D_\x$ a \emph{Bruhat operator} for $\x$. We often say ``$\x$ as a Bruhat operator'' instead of $D_\x$ since we are mainly interested in describing the cohomology of the affine flag variety as a subalgebra in $\homa$.\\

As Bruhat operators on the affine nilCoxeter algebra, we have the following relations between the operators $[ij]$.
\begin{enumerate}[(a)]
\item $[ij]^2=0$.\\
\item $[ij][kl]=[kl][ij]$ if $\overline{i},\overline{j},\overline{k},\overline{l}$ are all distinct. \\
\item For $i,j,k$ with distinct residues, $[ij][jk]=[jk][ik]+[ik][ij]$ and $[jk][ij]=[ik][jk]+[ij][ik]$.\\
\item For distinct $i,j$ with $\overline{i}\neq\overline{j}$, $\sum_{\overline{j'}=\overline{j''}=\overline{j}}[ij'][ij'']=0.$\\
\item $[i,j]=[i+n,j+n]$  \\
\end{enumerate}

Note that the relations (a)-(c) are analogous to those in the definition of the Fomin-Kirillov algebra, and proofs for these relations are similar. The relation (d) is an affine type $A$ analogue of the quadratic relation in the bracket algebra which is a generalization of the Fomin-Kirillov algebra to (classical) Coxeter groups \cite{KM03}. The relation (e) is obvious since we have $t_{i,j}=t_{i+n,j+n}$ as elements in the affine symmetric group. The quotient algebra of $\mathcal{E}$ modulo relations (a)-(e) is called the \emph{affine Fomin-Kirillov algebra} $\afk$. 

\begin{rmk}\label{infinitesum}
 In later sections we may need to add infinitely many elements $\x_i$ in $\mathcal{E}$ or in $\afk$. For elements $\x_i$ in $\mathcal{E}$ for $i$ in some index set $\mathcal{I}$, we define the sum $\sum_{i \in\mathcal{I}} \x_i$ as an element in $\mathcal{E}$ if for every $w\in \sn$, all but finitely many $A_w\cdot \x_i$ are zero. 

\end{rmk}
Note that the Bruhat action of $\afk$ on $\mathbb{A}$ is not faithful. One can show that for $x,y,z$ with distinct residues modulo $n$, we have $[xy][yz][xy]=[yz][xy][yz]$ and this relation is not implied by relations (a)-(e). For later use, define $\mathcal{A}'$ to be the quotient algebra of $\mathcal{A}$ modulo relations (a)-(c). \\

The following lemma is useful to prove that certain infinite expressions in $\mathcal{A}$ give a valid action, thus also in $\mathcal{E}$.
\begin{lemma} 
For a positive integer $M$, let $\mathcal{S}_M$ be the subset of $\mathcal{S}$ consisting of elements $[ij]$ with $0<j-i<M$. For an element $\x=\sum_{J,m} a_{J,m}\x_{J,m}$ in $\mathcal{A}$ if for any constant $M$, all but finitely many $a_{J,m}$ for $J\in \prod_{i=1}^m \mathcal{S}_M$ vanish, then the element $\x$ gives a valid action on $\mathbb{A}$.
\end{lemma}
\begin{proof} It is enough to show that the element $\x$ as an action on $A_w$ gives a valid element in $\mathbb{A}$. Note that all $A_v$'s appearing in $A_w \cdot \x$ satisfy $w>v$ where $>$ is the Bruhat order in $\sn$. There are only finitely many chains of Bruhat covers starting from $w$ to $v$ when we identify two covers $v_1\lessdot v_2$ corresponding to indices $(i,j)$ and $(i+n,j+n)$. Therefore, the set consisting of $j-i$ for all indices $(i,j)$ appearing in the Bruhat interval $[v,w]$ has an upper bound. One can set this upper bound to $M$ and apply the hypothesis to make $A_w \cdot \x$ a finite expression. The theorem follows. \end{proof}
\begin{cor}\label{valid}
For an element $\x=\sum_{J,m} a_{J,m}\x_{J,m}$ in $\mathcal{A}$, if there exist constants $a,b$ such that all $[ij]$'s appearing in the expression satisfy $a<j$ and $i<b$ then $\x$ is in $\mathcal{E}$.
\end{cor}
\begin{proof} Since we have only finitely many $[ij]$'s satisfying $a<j,i<b,j-i<M$ for fixed constants $a,b,M$, the corollary follows.\end{proof}

Note that there is a (left) $\sn$- action on $\afk$ (and $\mathcal{E}$) defined by
$$w[ij]=[w(i),w(j)]$$
for $w\in \sn$ and $[ij]\in \mathcal{S}$. Indeed, one can check that the two-sided ideal generated by relations (a)-(e) is invariant under the $\sn$-action.

\section{Dunkl elements and Murnaghan-Nakayama elements}

In this section, we define Dunkl elements and MN elements and investigate identities among these elements.\\
\subsection{Dunkl elements}\label{Dunkl}
For $i\in \mathbb{Z}$, a Dunkl element $\atheta_i$ can be defined in an analogous way to the definition of the Dunkl element in $\fk$ defined in Section \ref{finitefk}.
\begin{definition}
For $i\in \mathbb{Z}$, define a Dunkl element $\atheta_i$ by 
$\sum_{j \in \mathbb{Z},\overline{j}\neq\overline{i}}[ij]$.
\end{definition}
Note that all terms appearing in $\atheta_i$ are either of the form $[ij]$ for $i<j$, or $[ji]$ for $j<i$. Therefore, $\atheta_i$ satisfies the hypothesis in Corollary \ref{valid} with $a=i-1$ and $b=i+1$, thus it gives a valid action on $\mathbb{A}$. \\

In this subsection, we use relations (a)-(e) to prove properties of Dunkl elements.

\begin{thm}\label{comm}
$\atheta_i$'s commute with each other for all $i$. 
\end{thm}
\begin{proof} For fixed $i,j$ with distinct residues, there are 3 cases for terms appearing in $\atheta_i\atheta_j$. For each term $[ix][jy]$ (or $[jy][ix]$), let the support of the term be the set $\{i,x,j,y\}$ modulo $n$. If the cardinality of the support is 4, then we have $[ix][jy]=[jy][ix]$ by the relation (a). Hence the restriction of $\atheta_i\atheta_j-\atheta_j\atheta_i$ on any support with cardinality 4 vanishes.\\
If the cardinality of the support is 3, $\atheta_i\atheta_j-\atheta_j\atheta_i$ vanishes on the support $\{i,j,x\}$ from the identity
$$[ix][jx]+[ij][jx]+[ix][ji]=[jx][ix]+[jx][ij]+[ji][ix].$$
The above identity follows from the relation (b).\\
If the cardinality of the support is 2, the support must be $\{i,j\}$ modulo $n$. Then $\atheta_i\atheta_j-\atheta_j\atheta_i$ vanishes on the support by the relation (d).\end{proof}

The following lemma is useful to derive additional identities among $\atheta_i$.
\begin{lemma}\label{cycliceq}
Let $a,b_1,\ldots,b_m$ be distinct integers modulo $n$. Then we have
\begin{align*}
\sum& \Big([a b_1][a b_2]\ldots[a b_m] [a b'_1]+ [a b_2][a b_3]\ldots[a b_m][a b_1][a b'_2]\\
+&\cdots+[a b_m][a b_1]\ldots [a b_{m-1}][a b'_m]\Big)=0
\end{align*}
where the sum is over all integers $b'_i$ congruent to $b_i$ modulo $n$.
\end{lemma}
\begin{proof} We omit the proof since it is a simple generalization of \cite[Lemma 7.2]{FK99}, proved by induction on $m$. For $m=1$, the theorem follows from the relations (d).\end{proof}

We have the following corollaries of Lemma \ref{cycliceq}, which will be frequently used later.
\begin{cor}\label{distinctres}
$$\atheta_i^m=\sum [i a_1] \ldots [i a_m]$$
where the sum is over all integers $a_1,\ldots , a_m$ such that $a_1,\ldots,a_m,i$ have distinct residues modulo $n$.
\end{cor}

\begin{cor}\label{vanish}
If $m\geq n$, then $\atheta_i^m=0$ for all $i$.
\end{cor}

From now on, we will not need to use the relations (d) and (e) in the definition of $\afk$ after using Corollaries \ref{distinctres}, \ref{vanish}. 

\subsection{Murnaghan-Nakayama elements}\label{MN}

We define Murnaghan-Nakayama elements $\p_m(i)$ (MN elements in short) in $\afk$ as a generalization of $\theta_1^m+\cdots+\theta_i^m$ in $\fk$. Unlike the finite case, MN elements are not generated by Dunkl elements $\atheta_i$. We define MN elements by investigating the combinatorics of the Fomin-Kirillov algebra studied by M\'{e}sz\'{a}ros, Panova, Postnikov \cite{MPP14} and generalizing them to affine case. \\

Let $\mathcal{D}$ be the 2-dimensional infinite grid. A \emph{box} is specified by its position $(i,j)$ when the vertices of the box are $(i,j),(i,j+1),(i+1,j),(i+1,j+1)$. Let $\mathcal{D}_a$ be the set of all boxes at $(i,j)$ with $i\leq a <j$. A \emph{diagram} $D$ on $\mathcal{D}_a$ is a finite collection of boxes in $\mathcal{D}_a$. For a diagram $D$ on $\mathcal{D}_a$, we associate a graph with the vertex set $\mathbb{Z}$ obtained by adding an edge between $i$ and $j$ for each box at $(i,j)$ in $D$, then remove all vertices that are not adjacent to any edge. We say that a diagram $D$ is a \emph{connected tree} if the associated graph is a tree, and all vertices in the tree have distinct residues modulo $n$. Let $\supp(D)$ be the set consisting of indices of all vertices in the single tree in the associated graph for $D$ and $c(D)$ the number of vertices $(i,j)$ in the tree with $i\leq a$. See Figure 1 for an example of a connected tree and the associated graph. Note that the box at $(i,i+np)$ does not appear in a connected tree for any $i, p\in\mathbb{Z}$. \\

A \emph{labeling} $D_L$ on a diagram $D$ is a bijection from a set $\{1,2,\ldots,|D|\}$ to the set of boxes in $D$. For a labeling $L$ of a connected tree $D$, one can associate an element in the affine FK algebra defined by $x_{D_L}=[D_L(1)][D_L(2)]\ldots[D_L(|D|)]$ where $[D_L(i)]$ is $[a_i b_i]$ for the $i$-th box placed at $(a_i,b_i)$. We call two labelings $L$ and $L'$ equivalent if we have $x_{D_L}=x_{D_{L'}}$ by only using commutation relations.\\

\begin{center}
\begin{figure}
\includegraphics[width=4.8in]{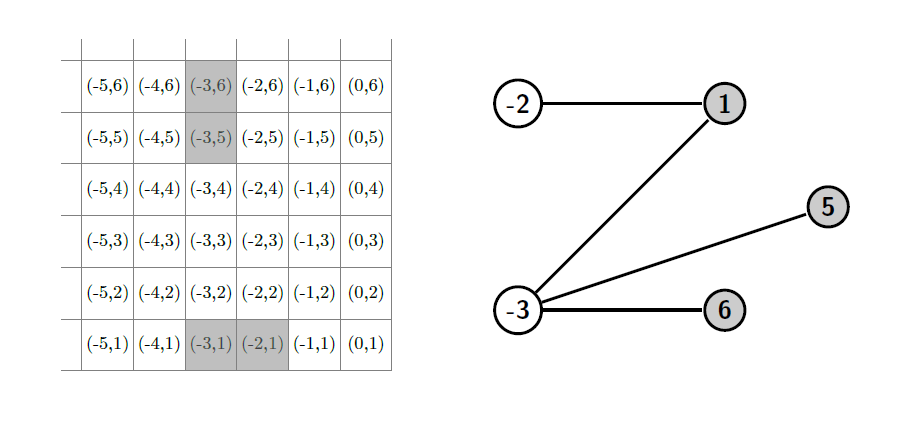}
\caption{A connected tree on $D_0$ for $n=6$. We color the box gray when the box is in the diagram. Note that the associated graph of a connected tree on $D_0$ is bipartite so that we color vertices with positive numbers gray and vertices with non-positive numbers white. We have $\supp(D)=\{-3,-2,1,5,6\}$ and $c(D)$ is $2$, the number of white vertices.}
\end{figure}
\end{center}

The following lemma is an obvious generalization of \cite[Lemma 7]{MPP14}.
\begin{lem}\label{postnikov}
Let $v,l$ be positive integers and $D$ be a connected tree in $\mathcal{D}_a$ with $l+v$ boxes contained in $l$ rows and $v+1$ columns. Then the following two sets are equal:
\begin{enumerate}
\item The equivalent classes of labelings of $D$ such that the class contains a labeling $(i_1,j_1), \ldots, (i_{l+v},j_{l+v})$ with: \\
$i_1,\ldots,i_l$ are distinct, $j_1\leq\cdots,\leq j_l$, $j_{l+1},\ldots,j_{l+v}$ are distinct, $i_{l+1}\leq \cdots \leq i_{l+v}$.

\item The equivalent classes of labelings of $D$ such that the class contains a labeling $(i_1,j_1), \ldots, (i_{l+v},j_{l+v})$ with: \\
$i_1,\ldots,i_{l-1}$ are distinct, $j_1\leq\cdots,\leq j_{l-1}$, $j_{l},\ldots,j_{l+v}$ are distinct, $i_{l}\leq \cdots \leq i_{l+v}$.
\end{enumerate}
\end{lem}
Let $M(D)=\{D_{L_1},\ldots,D_{L_h}\}$ be the set of representative labelings of equivalent classes in Lemma \ref{postnikov}. See Figure 2 for an example of a labeling satisfying Lemma \ref{postnikov}.

\begin{center}
\begin{figure}
\includegraphics[width=2.5in]{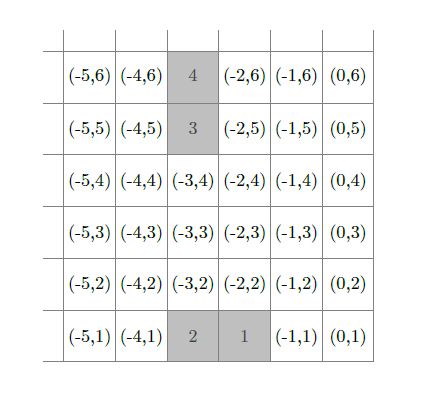}
\caption{The labeling of the gray diagram satisfies  Lemma \ref{postnikov} with $l=3,v=1$. We have $x_{D_L}=[-2,1][-3,1][-3,5][-3,6]$.}
\end{figure}
\end{center}

\begin{definition}\label{p}
Let $m$ and $a$ be positive integers. Define $\p_m(a)$ in $\afk$ by
$$\p_m(a)=\sum_{D\in \mathcal{D}_a}\sum_{D_L \in M(D)} (-1)^{c(D)-1}x_{D_L}$$
where the first sum runs over all connected trees in $\mathcal{D}_i$.

\end{definition}
Since all $[ij]$'s appearing in $\p_m(a)$ satisfy $i\leq a<j$, $\p_m(a)$ gives a valid action on the affine nilCoxeter algebra by Corollary \ref{valid}.
We denote $\p_m(0)$ by $\p_m$.
\begin{rmk}\label{pinf}
Roughly speaking, $\p_m(i)$ behaves similarly to the infinite summation $\sum_{j=-\infty}^i \atheta_j^m$. The expression $\sum_{j=-\infty}^i \atheta_j^m$ is not well-defined since we have relations $\atheta_i=\atheta_{i+n}$ and $\sum_{i=0}^{n-1} \atheta_i=0$ if we use all relations (a)-(e) defining affine Fomin-Kirillov algebra. Recall that $\mathcal{A}'$ is the quotient algebra of $\mathcal{A}$ modulo relations (a)-(c). Then $\p_m(a)$ as an element of $\mathcal{A}'$ is the same as 
$$\sum_{i=-\infty}^a \sum [i a_{i,1}] \ldots [i a_{i,m}]=\lim_{j\rightarrow -\infty}\sum_{i=j}^a \sum [i a_{i,1}] \ldots [i a_{i,m}]$$
where the second sum is over all integers $a_{i,1},\ldots a_{i,m}$ such that $a_{i,1},\ldots,a_{i,m},i$ have distinct residues modulo $n$.
Here we use $\sum [i a_{i,1}] \ldots [i a_{i,m}]$ instead of $\atheta_i^m$. However, the commutativity between $\p_m(a)$ and $\atheta_i$ does not follow if we do not use the relations (d)-(e). In fact, $\atheta_i$'s does not commute without the relation $(d)$.

\end{rmk}

There is a surjection from the subalgebra of $\mathcal{A}'$generated by $\atheta_i$ and $\p_m$ to the subalgebra of the finite FK algebra generated by Dunkl elements . For a fixed set $T=(a_1<\cdots<a_p)$ containing distinct integers modulo $n$, one can define the projection $r_T:\mathcal{A}'\rightarrow \mathcal{A}'$ defined by $[ij]\mapsto [ij]$ if $\{i,j\}\subset T$ and $0$ otherwise. For $T=\{1,2,\ldots,n\}$, this is the canonical surjection from $\mathcal{A}'$ to $\fk$. \\
One can check that $r_T(\atheta_i)=\theta_i$ for $1\leq i \leq n$. The restriction of $\p_m(i)$ to finite Fomin-Kirillov algebra via $r_T$ is the same as $p_m(\theta_1,\ldots,\theta_i)=\theta_1^m+\cdots+\theta_i^m$. Indeed, M\'{e}sz\'{a}ros, Panova and Postnikov gave a (positive) expression of the Schur function $s_\lambda(\theta_1,\ldots,\theta_i)$ in $\theta_1,\cdots,\theta_i$ for the hook shape $\lambda$ in \cite{MPP14}, and one can deduce the expression for $\theta_1^m+\cdots+\theta_i^m$ from the well-known identity
$$p_m=\sum_{i=0}^{m-1} (-1)^i s_{(m-i,1^i)}.$$

\subsection{More identities among Dunkl elements and MN elements.}
For $p<n$, let $T=(a_1<\cdots<a_p)$ be a set of distinct integers with distinct residues modulo $n$. Let $f_T$ be the injection from $[n]$ to $\mathbb{Z}$ defined by $f_T(i)=a_i$.
Then for such $f_T$, there is an injection from $\fk$ to $\mathcal{A}'$ defined by
$$[ij]\in\fk \mapsto [f_T(i)f_T(j)] \in \mathcal{A}'.$$
Since all relations among $[ij]\in \fk$ follow from relations (a)-(c) in $\afk$, this map is well-defined. Therefore, all relations between elements in $\fk$ also hold in $\mathcal{A}'$ and $\afk$ after applying the map $f_T$.

\begin{thm}\label{p+}
For $m<n, i \in \mathbb{Z}$, we have
$$\p_m(i-1)+\atheta_{i}^m=\p_m(i)$$
in $\afk$.
\end{thm}
\begin{proof} By Corollary \ref{distinctres} and Remark \ref{infinitesum}, it is enough to show that $\p_m(i-1)+\sum [i a_1] \ldots [i a_m]=\p_m(i)$ where the sum is over all integers $a_1,\ldots,a_m,i$ having distinct residues modulo $n$. In fact, we will show this identity in $\mathcal{A}'$.\\
For a fixed support $T=(a_1<\cdots<a_p)$, the restriction of the identity via $r_T$ is 
\begin{equation}\label{rt} r_T(\p_m(i-1))+\sum [i b_1] \ldots [i b_m]=r_T(\p_m(i))\end{equation}
on the support $T$, where the sum runs over all $\{i,b_1,\cdots,b_m\}=T$ with $m+1=p$. The identity (\ref{rt}) is the image of the identity $(\theta_1^m+\cdots+\theta_{i-1}^m)+\theta_{i}^m= (\theta_1^m+\cdots+\theta_i^m)$ via the map $f_T$, therefore the identity holds for each support $T$. After taking the summation of the identity (\ref{rt}) for all possible supports $T$, the theorem follows.\end{proof}

One can apply Theorem \ref{p+} to prove the following theorems.

\begin{thm} \label{sym0} For $m>0$, we have 
$$\sum_{i=1}^{n} \atheta_i^m=0.$$
\end{thm}
\begin{proof}
It is obvious from Theorem \ref{p+} and the fact $\p_m(0)=\p_m(n)$ by the relation (e).
\end{proof}
\begin{thm}\label{sp}For $i,a\in \mathbb{Z}/n\mathbb{Z}$,we have
\begin{align}s_i\p_m(a)=\begin{cases}\p_m(a) & \text{if } i\neq a \\ \p_m(i)+\atheta_{i+1}^m-\atheta_i^m &  i=a. \end{cases}\end{align}
\end{thm}
\begin{proof}
If $i$ and $a$ are distinct modulo $n$, one can show that $s_i\p_m(a)=\p_m(a)$ by Definition \ref{p}. Indeed, this follows from the fact that for a connected tree $D$ on $\mathcal{D}_a$ and $L\in M(D)$, $s_i(D)$ is also a connected tree on $\mathcal{D}_a$ and the labeling $s_i(L)$ satisfies Lemma \ref{postnikov} where $s_i(L)$ is the unique labeling on $s_i(D)$ satisfying $s_i(x_{D_L})=x_{s_i(D)_{s_i(L)}}$ without the commutation relation. When $i=a$, by Theorem \ref{p+} we have 
$$s_i\p_m(i)=s_i(\p_m(i-1)+\atheta_{i}^m)=\p_m(i-1)+\atheta_{i+1}^m=\p_m(i)+\atheta_{i+1}^m-\atheta_i^m.$$\end{proof}

\section{Relations between operators}\label{mainsection}
Recall that we abuse notation $D_u$ so that it means one of Bruhat operators or cap operators depending on $u$. For properties of cap operators, we will refer to \cite{Lee14}. Also recall that $D_{\p_m}$ is the Bruhat operator for $\p_m$. We call $D_{\p_m}$ a \emph{Murnaghan-Nakayama operator} of degree $m$ (a MN operator in short). For $0\leq i<m<n$, let $\rho_{i,m}$ be the element $s_{-i}s_{-i+1}\ldots s_{-1}s_{m-1-i}s_{m-2-i}\ldots s_1s_0$ in $\sn$. In this section, we prove the following theorem:

\begin{thm}\label{main}
For $0\leq i<m<n$, we have
\begin{align*}
D_{\atheta_i}&=D_{s_{i+1}}-D_{s_i},\\
D_{\p_m}&=\sum_{i=0}^{m-1}(-1)^i D_{\rho_{i,m}}
\end{align*}
as elements in $\homa$.
\end{thm}

Theorem \ref{main} allows us to identify $D_{\atheta_i}$ with $\xi^{s_{i+1}}-\xi^{s_i}$, and $D_{\p_m}$ with $\xi(m):= \sum_{i=0}^{m-1}(-1)^i \xi^{\rho_{i,m}}$. We later see that when we consider $\sum_{i=0}^{m-1}(-1)^i D_{\rho_{i,m}}$ as an element in $\Lambda^{(k)}\cong H^*(\gr)\subset H^*(\fl)$, it is the same as the power sum symmetric function $p_m$ so names for $D_{\p_m}$ and $\p_m$ are justified. Since $p_m$'s for $m=1,\ldots, n-1$ generate $H^*(\gr)$ and $\xi^{s_{i+1}}-\xi^{s_i}$'s generate the image $ev^*(H^*(Fl_n))$ in $H^*(\fl)$ by Theorem \ref{degree1image}, Dunkl elements $\atheta_i$ and MN elements $\p_m$ as an element in $\homa$ generate the cohomology of the affine flag variety $H^*(\fl)$.\\

Note that the first equation in Theorem \ref{main} is immediate from the Chevalley rule, so we mainly focus on properties of $D_{\p_m}$ in this paper. To prove Theorem \ref{main}, we prove certain identities for MN operators $D_{\p_m}$(Theorem \ref{skip0}, \ref{p}, \ref{hi}) and those for the alternate sum $\sum_{i=0}^{m-1}(-1)^i D_{\rho_{i,m}}$ of cap operators (Section \ref{identitycap}) that uniquely determine these operators.\\

Let $\mathbb{A}(\afk)$ be the smash product of $\mathbb{A}$ and $\afk$ defined by the \emph{equivariant Bruhat action}
\begin{align}A_w[ij]=\begin{cases} w([ij])A_w+A_{wt_{ij}} & \text{ if } \ell(wt_{ij})=\ell(w)-1 \\ w([ij])A_w & \text{ otherwise. }\end{cases}\end{align}

To distinguish this action with the Bruhat action defined in (2), we will use $A_w\cdot [ij]$ for the Bruhat action.
By $(7)$, an element in $\mathbb{A}(\afk)$ can be written in a standard form i.e. in the form
$$\sum_{w\in \sn} f_w A_w$$
where $f_w$ is in $\afk$. One can show that $\mathbb{A}(\afk)$ has a basis $\{A_w \mid w\in \sn\}$ over $\afk$.\\

The equivariant Bruhat action strictly contains all information on the Bruhat action. Let $\phi:\mathbb{A}(\afk) \rightarrow \mathbb{A}$ be the evaluation map at $0$ defined by sending all $[ij]$ to $0$. For example, $\phi(2[12][23]A_2+2[23]A_1+3A_0)=3A_0$. Then we have $\phi(A_w[ij])=A_w\cdot[ij]$ since $\phi([ij]A_w)=0$ for all $[ij]\in \mathcal{S}$. Note that the Bruhat action is an action on the affine nilCoxeter algebra $\mathbb{A}$, but the equivariant Bruhat action is an action on $\mathbb{A}(\afk)$. \\

\subsection{MN operators of degree one}
Let us consider $D_{\p_1(a)}(A_w)$ for $w \in \sn,a \in \mathbb{Z}$. From the definition of $\p_1(a)$ we have
\begin{align}\label{Awp1}
 D_{\p_1(a)}(A_w)=  \sum_{w\rightarrow u}A_u \end{align}
where the sum is over marked strong covers $w\rightarrow u$ with respect to $a$. It turns out this calculation also appears in the study of the \emph{affine nilHecke algebra} \cite{KK86,Lee14}. \\

\begin{lemma}\label{deg1}
For $w \in \sn, 0\leq a \leq n$, we have
$$D_{\p_1(a)}(A_w)=D_{s_a}(A_w)$$
where $D_{s_0}$ is the cap operator for $s_a$.
\end{lemma}
\begin{proof} The lemma follows from the Chevalley rule for the affine flag variety: 
$$\xi^w \xi^{s_a} = \sum_{u} \xi^u$$
where the sum is over all marked strong cover from $u$ to $w$ with respect to $a$. \end{proof}
\begin{rmk} \label{bssremark} The author showed in \cite{Lee14} that the operator $D_{\p_1(a)}$ is also the same as the BSS operators $D_{[1]}^{(a)}$ for $[1]$ defined by Berg, Saliola, Serrano \cite{BSS13}.

\end{rmk}
\subsection{MN operators of higher degree}
In this section, we prove Theorem \ref{skip0}, \ref{p}, \ref{hi} which uniquely determine the MN operators.

\begin{thm}\label{ap} 
For $a,j \in \mathbb{Z}/n\mathbb{Z}$ and $m<n$, we have 
$$A_a \p_m(a)=(\p_m(a)-\atheta_a^m)A_a+A_a \atheta_a^m.$$
$$A_j \p_m(a)=(\p_m(a))A_j.$$
where $j\neq a$ modulo $n$.
\end{thm}
\begin{proof} The second identity follows from Theorem \ref{sp}. By Theorem \ref{p+} and \ref{sp}, we have
$$A_a \p_i(a)=A_a (\p_i(a-1)+\atheta_a^m)=(\p_i(a)-\atheta_a^m)A_a+A_a \atheta_a^m.$$\end{proof}

The following theorem is a consequence of the second identity in Theorem \ref{ap}.
\begin{thm}\label{skip0}
For $w\in \sn$ and $w' \in S_n$, we have
$$D_{\p_m}(A_wA_{w'})=D_{\p_m}(A_w)A_{w'}.$$
\end{thm}

Let $\mathcal{C}$ be the subalgebra of $\afk$ generated by all $\atheta_i$'s and $\p_m$'s, and $\mathcal{C}_{Fl_n}$ be the subalgebra generated by Dunkl elements $\atheta_i$. Let $\ac$ be the subalgebra of $\mathbb{A}(\afk)$ generated by all $A_i$'s and $\mathcal{C}$, and let $\at$ be the subalgebra of $\mathbb{A}(\afk)$ generated by all $A_i$'s and $\mathcal{C}_{Fl_n}$. \\

We need the following lemmas to prove Theorem \ref{p}.
\begin{lem} For $w\in \sn$, $A_w \p_m - \p_m A_w$ lies in $\at$.
\end{lem}
\begin{proof}We use induction on $\ell(w)$. Let $w=vs_i $ with $\ell(w)=\ell(v)+1$ for some $0\leq i <n$ and assume that the theorem holds for $v$. Then
\begin{align*}
A_w\p_m-\p_m A_w&=A_v A_i \p_m - \p_m A_v A_i \\
 &= \begin{cases} A_v \p_m A_i - \p_m A_v A_i &\text{for } i\neq 0 \\ A_v \p_m A_0 - \p_m A_v A_0 + A_v ( A_0 \atheta_0^m - \atheta_0^m A_0)&\text{for } i=0. \end{cases}
\end{align*}
Note that we used Theorem \ref{ap} in the calculation. Since both $(A_v \p_m- \p_m A_v) A_i$ and $A_v ( A_a \atheta_a^m - \atheta_a^m A_a)$ lie in $\at$, the theorem follows. \end{proof}

One may write $A_w \p_m - \p_m A_w= \sum_{v} f_{v,m}^w A_v$ where $f_{v,m}^w$ is a (noncommutative) polynomial in $\atheta_i$ of degree $m-\ell(w)+\ell(v)$.

\begin{lem} \label{htheta0}
For $\h\in\mathbb{B}$, we have $\phi(\h \atheta_i)=\h\cdot\atheta_i=0$.
\end{lem}
\begin{proof} By Theorem \ref{p+}, it is enough to show that $\h\cdot \p_1(a)$ is independent of $a$. This follows from the fact that every element $b$ in $\mathbb{B}$ is invariant under the automorphism of $\mathbb{A}$ sending $A_i$ to $A_{i+1}$, since generators $\h_i$ are invariant. \end{proof}

\begin{thm}\label{p}
For $\h\in\mathbb{B}$ and $w\in\sn$,
$$D_{\p_m}(\h A_w)=D_{\p_m}(\h)A_w+ \h D_{\p_m}(A_w).$$
\end{thm}
\begin{proof} \begin{align*}
D_{\p_m}(\h A_w)&=\phi(\h A_w \p_m)\\
&=\phi(\h\p_mA_w)+\phi(\h \sum_{v} f_{v,m}^w A_v)\\
&=\phi(\h\p_m)A_w+\phi(\h \sum_{\substack{v\\ \ell(v)=\ell(w)-m}}f_{v,m}^w A_v)& (\text{By Lemma \ref{htheta0}})\\
&=D_{\p_m}(\h)A_w+ \h D_{\p_m}(A_w).
\end{align*}\end{proof}

\begin{thm}\label{hi}
For $1\leq m\leq i<n$ and $a\in \mathbb{Z}$, we have
$$D_{\p_m(a)}(\h_i)=\h_{i-m}.$$
\end{thm}
\begin{proof} It is enough to show the statement for $a=0$ by taking an automorphism of $\mathbb{A}$ (resp. $\afk$) by sending $A_i$ to $A_{i-a}$ (resp. $[ij]$ to $[i-a,j-a]$). Note that $\h_i$'s are invariant under the automorphism.\\

 For $0 \in J \subset I$ with $|J|=i$, let $J'$ be the connected subset $\{-j_1,-j_1+1,\ldots,0,1,\ldots,j_2\}$ of $J$ with maximal size containing $0$. The cyclically decreasing element $w_J$ can be written of the form $w_{J'}w_{J''}$ where $J$ is a disjoint union of $J'$ and $J''$. By the construction of $J'$ and $J''$, $w_{J'}$ and $w_{J''}$ commute with each other and $w_{J''}$ does not contain $s_0$. If $J$ does not contain $0$, we set $w_{J'}=id$ and $w_{J''}=w_{J}$. Note that $w_{J'}$ is simply $s_{j_2}s_{j_2-1}\ldots s_{1-j_1}s_{-j_1}$ and the inversion set of $w_{J'}$ is $\{ (-j_1,i) \mid -j_1 < i \leq j_2+1 \}$. \\\\
Let us calculate $A_{w_J}\cdot \p_m$. Since $w_{J''}$ does not contain $s_0$, we have
$$A_{w_J}\cdot \p_m=\left(A_{w_{J'}}A_{w_{J''}}\right)\cdot\p_m=\left(A_{w_{J'}}\cdot\p_m\right) A_{w_{J''}}.$$

Let $\x=[x_1y_1]\cdots[x_my_m]$ be a noncommutative monomial of degree $m$ in $\e$ satisfying $x_b \leq 0 <y_b $ for all $b$.
Then one can show that $A_{w_{J'}}\cdot\left([x_1y_1]\cdots[x_my_m]\right)$ does not vanish if and only if $x_b=-j_1$ for all $b$ and $j_2+1 \geq y_m >\cdots >y_1 \geq 0$. In this case, we have
$$A_{w_{J'}}\cdot\left([x_1y_1]\cdots[x_my_m]\right)= A_{w_{J/\{y_1,y_2,\ldots,y_m\}}}.$$
Consider the set of all pairs $\Omega=\{(A_{w_J},\x)\}$ where $J$ is a subset of $I$ with $|J|=i$ and $\x=[x_1y_1]\cdots[x_my_m]$ satisfying $x_b=-j_1$ for all $b$ and $j_2+1 \geq y_m >\cdots >y_1 \geq 0$. Define the map from $\Omega$ to the set consisting of $A_{v}$'s for all cyclically decreasing elements $v$ of length $i-m$, by sending $(A_{w_J},\x)$ to $A_{w_{J/\{y_1,y_2,\ldots,y_m\}}}$.\\

We claim that the map is bijective by providing an inverse map. For $K\subset I$ with $|K|=i-m$ and a nonnegative integer $p$, define $K_p$ to be the union of $K$ and $\{0,1,\ldots,p\}$. Let $l$ be the minimal number satisfying $|K_l|=m$. Then we set $J=K_l$, $x_b=-j_1$ for all $b$, and $\{y_1<\cdots<y_m\}=J/K$. One can show that both maps are inverse each other. Since all $\x$'s that occur in the set $\Omega$ appear with a coefficient 1 in $\p_m$, the theorem follows. \end{proof}

\subsection{Proof of Theorem \ref{main}}\label{identitycap}
 In this subsection, we introduce BSS operators \cite{BSS13,BSS14} defined by Berg, Saliola, Serrano. Then we prove Theorem \ref{main} by using the relation between BSS operators and cap operators and theorems in \cite{BSS13,BSS14}. \\
 
For $a \in \mathbb{Z}$, let $\mathcal{G}^{(a)}$ be the edge-labelled oriented graph defined in \cite{BSS14} with the affine symmetric group as a vertex set: there is an edge from $x$ to $y$ labelled by $y(j)=x(i)$ whenever $\ell(x)=\ell(y)+1$ and there exists $i\leq a <j$ such that $y t_{ij}=x$. For such a pair $(x,y)$, we denote by $y\overset{y(j)}\rightarrow x$ and call a marked strong cover in \cite{LLMS10}. Denote $\mathcal{G}^{(0)}$ by $\mathcal{G}$. The \emph{ascent composition} of a sequence $a_1,a_2,\ldots,a_m$ is the composition $[i_1,i_2-i_1,\ldots,i_j-i_{j-1},m-i_j]$, where $i_1<i_2<\cdots<i_j$ are the ascents of the sequence.\\

If $w_0\overset{a_1}\longrightarrow \cdots \overset{a_m}\longrightarrow w_m$ is a path in $\mathcal{G}^{(a)}$, we denote it by $\comp(w_0\overset{a_1}\longrightarrow \cdots \overset{a_m}\longrightarrow w_m)$ the ascent composition of the sequence of labels $a_1,\ldots,a_m$. For a composition $J=[j_1,\ldots,j_l]$ of positive integers, let $|J|$ be the sum $j_1+\cdots+j_l$. Let $D_J^{(a)}$ be the operator in \cite{BSS14} defined by
$$D_J^{(a)}(A_w)=\sum_{\comp(w_0\overset{a_1}\longrightarrow \cdots \overset{a_m}\longrightarrow w_m)=J} A_{w_m}.$$
where the sum runs over all path in $\mathcal{G}$ of length $|J|$ starting at $w=w_0$ whose sequence of labels has the ascent composition $J$. We call $D_J^{(a)}$ the \emph{BSS operator}. We denote $D_J^{(0)}$ by $D_J$. When we use the BSS operators $D_J^{(a)}$ in this paper, either $a=0$ or $J=[1]$. \\

Berg, Saliola and Serrano \cite{BSS13} showed that $D_{[1]}^{(a)}$ is the same as the cap operator $D_{s_{a}}$ for $0\leq a <n$ as mentioned in Remark \ref{bssremark}.
In \cite{Lee14}, the author proved that for a positive integer $m$, a BSS operator $D_{[m]}$ for $[m]$ is the same as the cap operator $D_{s_{m-1}s_{m-2}\ldots s_1s_0}$, which implies the Pieri rule for the cohomology of the affine flag variety. Note that we will abuse notation $D_u$ so that it means one of the above operators depending on whether $u$ is an element in the affine FK algebra, an element in $\sn$ or a composition. \\

For $0\leq i<m<n$, we defined $J_{i,m}$ to be a composition $[m-i,1^i]$. Also recall that $\rho_{i,m}$ is $s_{-i}s_{-i+1}\ldots s_{-1}s_{m-1-i}s_{m-2-i}\ldots s_1s_0$. 
Now we show that $D_{\rho_{i,m}}$ and $D_{J_{i,m}}$ are the same.
\begin{thm}\label{leebss}
$$D_{\rho_{i,m}}=D_{J_{i,m}}.$$
\end{thm}

\begin{proof}
Note that for $i=0,m-1$, it follows from the theorems in \cite{Lee14}. Indeed, for $i=0$ the theorem is the main theorem in \cite{Lee14}, and for $i=m-1$ one can show the theorem by applying the automorphism of $\sn$ and $\mathbb{A}$ sending $s_i$ (resp. $A_i$) to $s_{-i}$ (resp. $A_{-i}$).\\

 For a composition $J=[j_1,j_2,\ldots,j_l]$, observe from the definition of $D_J$ that
$$D_J=D_{j_1}\circ D_{[j_2,\ldots,j_l]} - D_{[j_1+j_2,\ldots,j_l]}.$$
By setting $J=J_{i,m}$, we have 
$$ D_{J_{i,m}}=D_{[m-i]}\circ D_{[1^i]} - D_{J_{i-1,m}}.$$
Since it is already known that $D_{[m-i]}=D_{\rho_{0,m-i}}$ and $D_{[1^i]}=D_{s_{-i+1}s_{-i+2}\cdots s_{-1}s_{0}} $, it is enough to show that 
$$ D_{\rho_{i,m}}=D_{\rho_{0,m-i}}\circ D_{s_{-i+1}s_{-i+2}\cdots s_{-1}s_{0}} - D_{\rho_{i-1,m}}.$$
Since $D_w$ can be identified with the Schubert class $\xi^w$, the above identity is equivalent to  
$$\xi^{\rho_{i,m}}=\xi^{\rho_{0,m-i}}\xi^{s_{-i+1}s_{-i+2}\cdots s_{-1}s_{0}} - \xi^{\rho_{i-1,m}}.$$\\
The above equality can be shown in $\Lambda^{(k)}\cong H^*(\gr)$, since each Schubert class appearing above corresponds to a Schur function. In fact, the calculation of Schur functions $s_{J_{0,m-i}}s_{J_{i-1,i}}=s_{J_{i,m}}+s_{J_{i-1,m}}$ implies the above identity.
\end{proof}\\

{\bf \emph{Proof of Theorem \ref{main}}.}
We show Theorem \ref{main} by proving the identity
$$D_{\p_m}=\sum_{i=0}^{m-1}(-1)^i D_{J_{i,m}}.$$
First of all, they both satisfy the following lemma:
\begin{lem} For $D=D_{\p_m}$ or $D_{J_{i,m}}$, $w\in \sn$ and a $0$-Grassmannian element $v$, we have
$$D(A_wA_v)=D(A_w)A_v.$$
\end{lem}
\begin{proof} For $D=D_{\p_m}$ it follows from Theorem \ref{skip0}, and for $D=D_{J_{i,m}}$ it follows from the fact that $D_{J_{i,m}}$ are generated by $D_{[a]}$ for $a\geq 1$ and the lemma holds for $J=[a]=J_{0,a}$ by \cite[Theorem 4.8]{BSS14}.\end{proof}

For any composition $J$, the restriction of $D_J$ to $\mathbb{B}$ is $\overline{s_J}^\perp$ where $s_J$ is the ribbon Schur function indexed by $J$ \cite[Theorem 4.9]{BSS14}. By letting $J=J_{i,m}$, the restriction of $D_{J_{i.m}}$ is $\overline{s_{J_{i,m}}}^\perp$, where $s_{J_{i,m}}$ is the Schur function for the hook shape $[m-i,1^i]$.\\
Recall the following theorems about the power sum symmetric functions $p_m$: (see \cite{Sta99} for instance):
$$p_m=\sum_{i=0}^{m-1}(-1)^i s_{J_{i,m}},$$
$$p_m^\perp(fg)=p_m^\perp(f)g+f p_m^\perp(g),$$
$$p_m^\perp(h_i)=h_{i-m}$$
for any symmetric functions $f,g$. Therefore, we proved that $D=\sum_{i=0}^{m-1}(-1)^i D_{J_{i,m}}$ satisfies the following identities.
\begin{enumerate}
\item $D(fg)=D(f)g+f D(g)$ for $f,g \in \mathbb{B}.$
\item $D(hA_v)=D(h)A_v$ for $h\in \mathbb{B}$ and a 0-Grassmannian element $v$.
\item $D(\h_i)=\h_{i-m}$
\end{enumerate}
Note that the above identities uniquely determine $D$ as an action on $\mathbb{A}$. Since $\p_m$ also satisfies the above identities by Theorems  \ref{skip0}, \ref{p}, \ref{hi}, the theorem follows. \qed \\

\section{Divided difference operators}\label{divided}
In this section, we define the divided difference operators on the affine FK algebra and describe the connection between the divided difference operators and operators defined by Kostant and Kumar \cite{KK86} . \\
\subsection{Operators defined by Kostant and Kumar and the affine nilCoxeter algebra}
The operator $\partial_i : H^*(\fl) \rightarrow H^*(\fl)$ is defined as follows by Kostant and Kumar in \cite{KK86}. 

\begin{definition} \label{KK}For $i\in \mathbb{Z}/n\mathbb{Z}$, let $\xi^w$ be the Schubert class for $w$ in $H^*(\fl)$. Then  
$$\partial_{i}\xi^w = \begin{cases} \xi^{ws_i}& \text{ if } \ell(ws_i)= \ell(w) -1,\\
 0 &\text{ otherwise.}\end{cases}$$
\end{definition}

For $w =s_{i_1}\ldots s_{i_l}$, $\partial_w$ is defined by $\partial_{i_1}\ldots \partial_{i_l}$. Note that in \cite{KK86}, $\partial_w$ is denoted by $A_w$ and this is how they defined nilHecke algebra and nilCoxeter algebra. Therefore, there is an action of $\mathbb{A}$ on $H^*(\fl)$ defined by Definition \ref{KK}.\\
When we consider $f$ in $H^*(\fl)$ as an element in $\homq$, we denote $\partial_w f$ by $A_i\bullet f$, and we call this action $\bullet$-action. Kostant and Kumar showed that
$$(A_i\bullet f)(A_w)= f(A_wA_i).$$
See \cite{KK86}, \cite[Chapter 4]{LLMSSZ} for details. 

\subsection{Divided difference operators on $\afk$.}
Recall that for $w \in \sn$, the action of $w$ on $\afk$ is given by $w [ij] = [w(i) w(j)]$. \\

For integers $i<j$ with distinct residues mod $n$, let $\Delta_{ij}$ be the divided difference operator as a left action on $\afk$ defined by the following:
\begin{enumerate}
\item For $\x, \y \in \afk$, $\Delta_{ij}(\x\y)=\Delta_{ij}(\x)\y+t_{ij}(\x)\Delta_{ij}(\y)$.\\
\item $\Delta_{ij} ([ab]) =1$ when $[ij]=[ab]$ in $\afk$, and $0$ otherwise.

\end{enumerate}
Note that the definition of $\Delta_{ij}$ generalizes the operators denoted by the same in \cite{FK99}. The divided difference operator $\Delta_{ij}$ is well-defined because $\Delta_{ij}$ stabilizes the two-sided ideal generated by relations (a)-(e), which is easy to check. \\

Recall the map $D: \afk \rightarrow \homa$ defined by the Bruhat operator, and $\mathcal{C}$ is the subalgebra of $\afk$ generated by $\atheta_i$'s and $\p_m$'s. Since we have $\mathcal{C} \subset \afk$, we have the restriction map $D:\mathcal{C} \rightarrow H^*(\fl) \subset \homa$. Note that the map $D:\mathcal{C}\rightarrow H^*(\fl)$ is surjective, because of Theorem \ref{main} and the discussion below Theorem \ref{main}.\\
 
Let $\phi_{id,*}$ be the surjection from $\homa$ to $\homq$ induced from the surjection $\mathbb{A}\rightarrow \mathbb{Q}$. By applying $\phi_{id,*}$, one can define a map $D^{id}:=\phi_{id,*}\circ D: \mathcal{C} \rightarrow \homq \cong H^*(\fl)$ defined by $D^{id}_\x=\phi_{id,*}\circ D_\x$ for $\x \in \mathcal{C}$. For the rest of the section, we show that 

\begin{thm} \label{ddosame}
The induced operator acting on $\homq$ from the divided difference operator $\Delta_i:= \Delta_{i,i+1}$ on $\mathcal{C}$ is the same as the action $A_i\bullet$. 
\end{thm}

\begin{proof}
We first need the following lemma:
\begin{lem} \label{dist}
 Let $w \in \sn$ and $\x \in \afk$. In $\mathbb{A}(\afk)$, we have
$$A_i \x = s_i(\x) A_i +  \Delta_{i}(\x).$$
\end{lem}
The lemma follows from the definition of the equivariant Bruhat action. Then Theorem \ref{ddosame} is equivalent to the identity
$$ D^{id}_{\Delta_i(\x)}(A_w)= D^{id}_{\x}(A_wA_i)=\left(A_i\bullet D^{id}_{\x}\right)(A_w).$$

By Lemma \ref{dist}, we have

$$ D_{\Delta_i(\x)}(A_w)=  D_{\x}(A_wA_i) -D_{s_i(\x)}(A_w)A_i.$$
By applying $\phi_{id,*}$, the theorem follows because $\phi_{id,*}(D_{s_i(\x)}(A_w)A_i)=0$.
\end{proof}

\section{Affine Schubert polynomial}
In this section, we define the affine Schubert polynomials using the divided difference operators.\\

Recall the isomorphisms 
\begin{align*} H^*(\fl)&\cong  H^*(\gr) \otimes_{\mathbb{Q}} H^*(Fl_n)\\
&\cong R_n:=\Lambda^{(k)} \otimes \mathbb{Q}[x_1,\cdots,x_{n}]/  \{ s_i (x)=0 \quad \forall i \}.
\end{align*}
Recall that $\Lambda^{(k)}$ is generated by $\overline{p_m}$, which we denote by $p_m$. Under the isomorphism $\homq \cong H^*(\fl) \cong R_n$, the operator $D_{\atheta_i}$ sends to $x_{i}$ since Schubert polynomial for $\xi^{s_i}_{Fl_n}$ is $x_1+\cdots+x_i$ in $ \mathbb{Q}[x_1,\cdots,x_{n}]/  \{ s_i (x)=0 \quad \forall i \}$ and $ev^*(\xi^{s_i}_{Fl_n})= \xi^{s_{i}}-\xi^{s_0}$ by Theorem \ref{degree1image}. Moreover, under the same isomorphism, the MN operator $D_{\p_m}$ sends to $p_m$. Since we defined the divided difference operators $\partial_i$ on $H^*(\fl)$, the same $\partial_i$ gives an action on $R_n$, and by translating identities about how $s_i$ and $\Delta_i$ act on $\atheta_i$ and $p_m$ provides Definition \ref{bgg}. To be more precise, the following implies Definition \ref{bgg}.

\begin{thm}
For $i\in \mathbb{Z}/n\mathbb{Z}$, the Weyl group action $s_i$ and $\Delta_i$ on $\mathcal{C}$ can be uniquely defined by the following properties.
\begin{enumerate}
\item For $f,g \in \mathcal{C}$, we have $s_i(fg)=s_i(f)s_i(g)$ and $\Delta_i$ satisfies the Leibniz's rule: for $f,g \in \mathcal{C}$, we have
$$\Delta_i(fg)= \Delta_i(f)g + s_i(f) \Delta_i(g).$$
\item For nonzero $i=1,\cdots, n-1$ and for all $m$, we have 
\begin{align*}
s_i(\p_m)&=\p_m \\
\Delta_i(\p_m)&=0.
\end{align*}
\item For $i=0$, we have
\begin{align*}
s_0(\p_m)&= \p_m+\atheta_1^m-\atheta_0^m\\
\Delta_0(p_m)&= \sum_{j=0}^{m-1} \atheta_1^{m-1-j}\atheta_0^j
\end{align*}

\item For all $i,j\in \mathbb{Z}/n\mathbb{Z}$, we have 
\begin{align*}
s_i(\atheta_j)&=\atheta_{s_i(j)} \\
\Delta_i(\atheta_j)&=\delta_{ij}-\delta_{i,j+1}.
\end{align*}

\end{enumerate}
\end{thm}
\begin{proof}
$(1)$ is clear from the definition of $s_i$ and $\Delta_i$. The first identity in $(2)$ is proved in Theorem \ref{sp} and the second identity $\Delta_i(\p_m)=0$ is true since there is no $[i,i+1]$ in any term in $\p_m$ and every $[i,j]$ in each term of $\p_m$ satisfies $i\leq 0 < i+1$. Recall that $i \neq 0$ modulo $n$ for $(2)$.\\

$(4)$ is obvious from the definition of $\atheta_i$. For $(3)$, the first identity directly follows from Theorem \ref{sp}, and we have
$$\Delta_{0}(\p_m)=\Delta_{0}(\p_m(1)-\atheta_1^m)=\Delta_{0}(-\atheta_1^m)=\sum_{j=0}^{m-1} \atheta_0^{m-1-j} \atheta_1^j$$
so the second identity follows. The last equality $\Delta_{0}(-\atheta_1^m)=\sum_{j=0}^{m-1} \atheta_0^{m-1-j} \atheta_1^j$ follows from both $(1)$ and $(4)$.
\end{proof}\\

  Note that when we act the usual divided difference operator $\partial_0= {1-s_0 \over x_0-x_1}$ on $p_m=\cdots + x_{-2}^m+ x_{-1}^m+x_{0}^m$, we get the same formula. Therefore, one should think of elements in $\Lambda^{(k)}$ as symmetric functions in variables $x_0,x_{-1},x_2, \ldots$ so that Definition \ref{bgg} becomes very natural. This idea is also related to Remark \ref{pinf} that $\p_m(i)$ are ``roughly" the same as $\sum^i_{j=-\infty} \atheta_j^m$.  \\

By defining divided difference operators in $R_n$ compatible with the operator $\partial_i$ (or $A_i$) defined by Kostant and Kumar, we obtain polynomial representatives of Schubert classes $\xi^w$:
 
\begin{thm}[Theorem \ref{asp}]
For $w \in \sn$, there exists a unique homogeneous element $\asp_w$ of degree $\ell(w)$ in $R_n$ satisfying
$$\partial_i\asp_w=\begin{cases} \asp_{ws_i} & \text{ if } \ell(ws_i)=\ell(w)-1 \\ 0 & \text{ otherwise}.\end{cases}$$
for $i\in \mathbb{Z}/n\mathbb{Z}$, with the initial condition $\asp_{id}=1$. Moreover, under the isomorphism $ H^*(\fl) \cong R_n$, a Schubert class $\xi^w$ sends to $\asp_w$. We call $\asp_w$ an affine Schubert polynomial for $w$.
\end{thm}

Note that the divided difference operators $\partial_i$ for nonzero $i$ restricted on $R_{Fl_n}=\mathbb{Q}[x_1,\cdots,x_{n}]/  \{ s_i (x)=0 \quad \forall i \}$ is the same as the divided difference operators defined by Lascoux and Sch\"{u}tzenberger \cite{LS82}. This implies that the affine Schubert polynomial for $w \in S_n$ is just the Schubert polynomial for $w$. Moreover, the affine Schubert polynomial $\asp_w$ for $0$-Grassmannian element $w$ is the same as the affine Schur functions, and for $w \in \sn$ the projection from $R_n$ to $R_{\gr}$ sends $\asp_w$ to the affine Stanley symmetric functions $\tilde{F}_w$ \cite{Lam08}, so Theorem \ref{restriction} follows.\\

The affine Schubert polynomials can be computed from the affine Schur functions. For $w \in \sn$, let $v$ be an element in $\sn$ such that $wv$ is 0-Grassmannian with $\ell(wv)=\ell(w)+\ell(v)$. There is always such a $v$ for any $w$ \cite{MS15}. Let $\f_{wv}$ be the affine Schur function for $wv$.
Then we have
\begin{align*}
\asp_w&= \partial_{v^{-1}} \asp_{wv}=\partial_{v^{-1}}\f_{wv}.
\end{align*}
Note that there is a formula for the expansion of the affine Schur functions in terms of power sum symmetric functions \cite{BSZ10}, so that one can compute the affine Schubert polynomials from Definition \ref{bgg}, \ref{asp}. One can also write a monomial expansion of affine Schubert polynomials, and this will appear in \cite{LLS15}.

\subsection{Examples}
$n=2$ case: For a positive integer $a$ and $i \in \{0,1\}$, define $w_{a,i}$ be the unique $i$-Grassmannian element of length $a$. Then
\begin{align*}
\asp_{w_{a,0}}&= m_{1,1, \ldots,1}= { p_1^n \over n!}\\
\asp_{w_{a,1}}&= { p_1^n \over n!}+{{p_1^{n-1}\over (n-1)!} x_1}
\end{align*}
$n=3$ case:
\begin{align*}
\asp_{id}&=1\\
\asp_{s_0}&=  p_1\\
\asp_{s_1}&=  p_1+x_1 \\
\asp_{s_2}&=  p_1+x_1+x_2\\
\asp_{s_1s_0}&= \frac{1}{2}(p_1^2+p_2)\\
\asp_{s_2s_1}&=\frac{1}{2} ((p_1+x_1)^2+(p_2+x_1^2))\\
\asp_{s_2s_1s_0}&= \frac{1}{3} p_3+ \frac{1}{2} p_2 p_1 + \frac{1}{6} p_1^3.
\end{align*}

\section{Nonnegativity conjecture}

In this section, we discuss the nonnegativity conjecture for Schubert polynomials in the Fomin-Kirillov algebra \cite{FK99}.\\

Although arguments in Section 6 show that the map $D$ from $\mathcal{C}$ to $H^*(\fl)$ is surjective, there is strong evidence that $D$ is in fact an isomorphism. First of all, all relations between Bruhat operators $D_{\atheta_i}$ for Dunkl elements are proved at the level of the affine FK algebra, namely the commutativity between Dunkl elements (Theorem \ref{comm}) and the vanishing of symmetric functions evaluated at Dunkl elements (Theorem \ref{sym0}). To show that $D$ is an isomorphism, we need to show the following conjecture.

\begin{conj}\label{com}
The commutativity between Dunkl elements and MN elements follows from the relations (a)-(e). This implies that the map $D: \mathcal{C}\cong H^*(\fl)$ is an isomorphism.
\end{conj}

A strong evidence is that the MN element $\p_m(i)$ should be considered as $\sum_{j=-\infty}^i \atheta_j^m$ secretly (see Remark \ref{pinf}). The infinite summation $\sum_{j=-\infty}^i \atheta_j^m$ does not make sense in $\afk$ but make sense in $\mathcal{A}'$, since the relations (a)-(c) are compatible with the topology of $B$ (Definition \ref{fk}) but relations (d)-(e) are not compatible. If there is a way to make the expression $\sum_{j=-\infty}^i \atheta_j^m$ valid while keeping the commutativity between Dunkl elements $\atheta_i$, Conjecture \ref{com} automatically follows.\\

Without the above argument, it is hard to avoid a case-by-case proof for proving the commutativity of those elements from the relations (a)-(e). Note that Berg et. al. showed that BSS operators commute in \cite{BSS14} but their result relies heavily on the affine insertion \cite{LLMS10} which involves complicated case-by-case algorithm and proofs. \\

Let $\asp_w^{\afk}$ be the \emph{affine Schubert element} in $\mathcal{C}$ corresponding to the Schubert class $\xi^w$, assuming Conjecture \ref{com}. Then one can generalize the nonnegativity conjecture in \cite{FK99}.
\begin{conj}[Nonnegativity conjecture for the affine Schubert elements]\label{nonn}
For $w \in \sn$, the affine Schubert element $\asp_w^{\afk}$ can be written as a nonnegative linear combination of noncommutative monomials in $\afk$. Moreover, there exists a linear combination of noncommutative monomials with nonnegative integer coefficients.
\end{conj}
Conjecture \ref{nonn} not only implies the nonnegativity conjecture in \cite{FK99}, but also implies the nonnegativity conjecture for the affine Schur elements and for the affine Stanley elements in $\mathcal{C}$ corresponding to affine Schur functions and affine Stanley symmetric functions respectively. Indeed, for $0$-Grassmannian element $w$ the affine Schubert element is the affine Schur element, and it is known that the affine Stanley symmetric functions are nonnegative linear combination of affine Schur functions \cite{Lam11}.\\

A combinatorially nonnegative formula for the affine Schubert element would provide a combinatorial rule for the structure constants of the affine Schubert polynomials, as well as Littlewood-Richardson coefficients for the flag variety. Assume that Conjecture \ref{nonn} is true for $u \in \sn$. Then $\asp_u^{\afk}$ can be written as $\sum a_\x \x$ where $\x$'s are noncommutative monomials and $a_\x\geq 0$. Since Bruhat operators and cap operators can be identified, we have
$$\sum_{\x} a_\x (A_w\cdot \x)= D_{\asp_u^{\afk}}(A_w)=D_u(A_w)= \sum_{v\in\sn} p^w_{v,u} A_v.$$
for $w\in \sn$. Then the structure coefficient $p^w_{v,u}$ is the same as the sum of $a_\x$ where $A_w \cdot \x =A_v$ which is nonnegative by Conjecture \ref{nonn}.

\section{Murnaghan-Nakayama rule for the affine flag variety}
In this section, we define a $k$-strong-ribbon and use them to combinatorially state the Murnaghan-Nakayama rule (MN rule in short) for the affine flag variety and the affine Stanley symmetric functions. Strong strips appearing in this section are strong strips with respect to $0$.\\

For $w,u\in \sn$ with $w>u$, consider a chain of marked strong covers from $w$ to $u$
$$ S=\big(w=w_0 \overset{(x_1,y_1)}\longrightarrow w_1 \overset{(x_2,y_2)}\longrightarrow \cdots \overset{(x_m,y_m)}\longrightarrow w_m=u\big).$$
Let $l$ be the index of the chain defined by $\left((x_1,y_1),(x_2,y_2),\ldots,(x_m,y_m)\right)$ and $\x_l$ be the corresponding element $[x_1,y_1][x_2,y_2]\cdots[x_m,y_m]$ in the affine FK algebra. We denote the chain by $w\overset{l}\leadsto u$.\\

For $w,u\in \sn$ with $w>u$, the chain $S=(w\overset{l}\leadsto u)$ is called a \emph{$k$-strong-ribbon} when $\x_l$ appears as a term in the definition of $\p_m$, i.e., there exists a connected tree $D$ on $\mathcal{D}_0$ and a labeling $L\in M(D)$ such that $\x_{l}=\x_{D_L}$. In this setup, we define the \emph{sign} $\sigma(D)$ of the $k$-strong-ribbon by $(-1)^{c(D)-1}$. Since $\x_{D_L}$ uniquely determines the diagram $D$, we use $\sigma(l)$ instead of $\sigma(D)$. We let inside$(S)=w$, outside$(S)=u$\\
For $w,v \in \sn$ and $m<n$, let $c^w_{m,v}$ be the number $\sum\sigma(l)$ where the sum runs over all $k$-strong-ribbons $(w\overset{l}\leadsto u)$. One can show from the definition that $D_{\p_m}(A_w)=\sum_u c^w_{m,u} A_u$. \\

Recall that $\xi^w$ is the Schubert class for $w$ in the cohomology of the affine flag variety and $\xi(m)=\sum_{i=0}^m (-1)^i \xi^{\rho_{i,m}}$. Note that $\xi(m)$ maps to $p_m$ via the map $pr_1^*: H^*(\fl)\rightarrow H^*(Gr)\cong \Lambda^{(k)}$. \\

\begin{thm}\label{mnaffine}
For $v\in \sn$, we have
$$\xi(m)\xi^v=\sum_{w\in\sn} c^w_{m,v} \xi^w=\sum \sigma(l)\xi^w$$
where the second sum runs over all $k$-strong-ribbons $(w\overset{l}\leadsto u)$ from $w$ to $u$.
\end{thm}
\begin{proof} For $u\in\sn$, recall that $\xi_u$ is the Schubert class for $u$ in the homology of the affine flag variety and let $\langle \cdot,\cdot \rangle$ be the pairing between the cohomology and homology of the affine flag variety. Then we have
$$\langle\xi(m)\cup \xi^v,\xi_w\rangle=\langle\xi^v,D_{\p_m}(\xi_w)\rangle=c^w_{m,v}.$$
\end{proof}

One can also obtain the MN rule for the affine Stanley symmetric functions. Lam showed that the Stanley symmetric function $\tilde{F}_w$ is the pullback $pr_1^*(\xi^w)$. By applying the pullback $pr_1^*$ to both sides of Theorem \ref{mnaffine}, we have the following MN rule.

\begin{cor}\label{mnstanley}
$$p_m\tilde{F}_v=\sum_{w\in\sn} c^w_{m,v}\tilde{F}_w=\sum \sigma(l) \tilde{F}_w$$
where the second sum runs over all $k$-strong-ribbons $(w\overset{l}\leadsto v)$ from $w$ to $v$.
\end{cor}

\begin{example}
Consider the identity $\f_{10}p_3=\f_{12310}-\f_{20310}+\f_{03210}$. Each term can be computed from the Bruhat actions of the following terms in $\p_m(0)$.
\begin{align*}
s_1s_2s_3s_1s_0\cdot [-2,1][-4,1][-1,1] &= s_1s_0\\
s_2s_0s_3s_1s_0\cdot [-4,1][-1,2][-1,1] &= s_1s_0\\
s_0s_3s_2s_1s_0\cdot [0,6][0,5][0,3] &= s_1s_0.
\end{align*}
and the corresponding $c(l)$, the number of non-positive integers appearing in the support, are $3,2,1$ respectively and the corresponding signs $\sigma(l)$ are $1,-1,1$ respectively. 

\end{example}

\section{Concluding remark}

The connection between MN operators and BSS operators enables us to make a marking-free definition of strong strips and $k$-Schur functions. Indeed, since the MN element $\p_m(0)$ as an action on $\mathbb{B}$ is the same as $\overline{p_m}$ and $D_{\rho_{0,m}}$ is the same as $\overline{h_m}$, one can deduce the relation between $D_{\p_m}$ and $D_{\rho_{0,m}}$ from the relation between $p_m$ and $h_m$ in the ring of symmetric functions. Note that the formula for $D_{\rho_{0,m}}$ as an element in the affine FK algebra is much more complicated than the definition of $\p_m$. Indeed, this situation also happens in the finite Fomin-Kirillov algebra. The fact that  $\p_m$ has fewer terms than other elements of the same degree makes computation substantially simpler. \\

One can obtain a formula of $k$-Schur functions in terms of power sum symmetric functions, from the MN rule for the affine Stanley symmetric functions. There are several definitions of $k$-Schur functions but in this paper the $k$-Schur functions are defined in terms of strong strips \cite{LLMS10}. It is proved in \cite{LLMS10} that the $k$-Schur function for $\lambda$ is dual to the affine Stanley symmetric functions indexed by $w(\lambda)$, where $w(\lambda)$ is the $0$-Grassmannian element corresponding to the $k$-bounded partition $\lambda$.  \\

A \emph{$k$-strong-ribbon tableau} is a sequence $T=(S_1,S_2,\ldots)$ of $k$-strong-ribbons $S_i$ such that outside$(S_j)=$ inside$(S_{j+1})$ for all $j \in \mathbb{Z}_{>0}$ and size$(S_i)=0$ for all sufficiently large $i$. We define inside$(T)=$ inside$(S_1)$ and outside$(T)=$ outside$(S_i)$ for $i$ large. The \emph{weight} $wt(T)$ of $T$ is the sequence 
\[wt(T)=(\text{size}(S_1),\text{size}(S_2),\ldots).\]
Let $\sigma(T)$ be the product of $\sigma(l_i)$ for all $i$ where $l_i$ is the index of $S_i$. \\

By applying Theorem \ref{mnstanley} repeatedly from $v=id$, one can deduce
$$p_\lambda=\sum_{v,T} \sigma(T)\tilde{F}_v$$
where the sum runs over all $k$-strong-ribbon tableaux $T$ starting from $v$ to $id$ with weight $\lambda$. By dualizing the equality, we can get the formula for $k$-Schur functions in terms of $p_m$.\\
Consider $\langle s^{(k)}_u, p_\lambda \rangle$ where $\langle \cdot, \cdot \rangle$ is the Hall inner product on $\Lambda_{(k)}\times \Lambda^{(k)}$. We have
\begin{align*}
\langle s^{(k)}_u, p_\lambda \rangle&=\langle s^{(k)}_u, \sum_T \sigma(T)\tilde{F}_v \rangle=\sum \sigma(T).
\end{align*}
where the last sum runs over all $k$-strong-ribbon tableaux $T$ from $u$ to $id$ of weight $\lambda$.
By using the identity $\langle p_\lambda, p_\mu \rangle=\delta_{\lambda\mu}z_\lambda$ where $z_\lambda=prod_i \alpha_i ! i^{\alpha_i}$ for $\lambda=(1^{\alpha_1},2^{\alpha_2},\ldots)$, we have

\begin{thm}\label{kschur}
$$s^{(k)}_u=\sum_{\lambda,T} {\sigma(T)\over z_\lambda} p_\lambda$$
where the sum runs over all $k$-strong-ribbon tableaux $T$ from $u$ to $id$ of weight $\lambda$.
\end{thm}
It would be interesting to generalize Theorem \ref{kschur} to define $t$-dependent $k$-Schur functions by introducing $t$-dependent $k$-strong-ribbon tableaux.\\

 Note that the definition of $k$-Schur functions using strong strips depends on markings, which was one of the major obstacles to defining strong strips as elements in the affine FK algebra. Recall that the marking of the strong cover $(u\overset{(i,j)}\longrightarrow w)$ is defined by $w(j)$. On element $[ij]$ does not encode the marking $w(j)$ since the element $[ij]$ is defined independent of $w$. When we take Theorem \ref{kschur} as the definition of the $k$-Schur functions, it is easier to prove that the $k$-Schur function is symmetric compared to the definition of $k$-Schur functions in terms of strong strips. However, it is not clear from Theorem \ref{kschur} that the $k$-Schur functions are monomial positive. \\

Since the $k$-Schur function $s^{(k)}_u$ is known to be Schur-positive \cite{LLMS13}, there exists a $S_n$-representation whose Frobenius image is $s^{(k)}_u$. By Theorem \ref{kschur}, the characters of this representation is $\sum \sigma(T)$ where $T$ runs over all $k$-strong-ribbon tableaux $T$ from $u$ to $id$ of weight $\lambda$. Note that Chen and Haiman \cite{CH08} conjecturally constructed $t$-graded $S_n$-representation whose Frobenius image is the involution of the $t$-dependent $k$-Schur function $\omega s^{(k)}_u[X;t]$, where the involution $\omega$ acting on $\Lambda$ by sending $s_\lambda$ to $s_{\lambda'}$. It would be interesting if the combinatorics studied by Chen and Haiman is related to the $k$-strong-ribbon tableaux in this paper.

\section*{Acknowledgments}
I thank Thomas Lam, Leonardo Mihalcea, Mark Shimozono for helpful discussions. I also thank an anonymous referee for helpful comments.


\begin{thebibliography}{99}

\bibitem{BH95} S. Billey and M. Haiman, {\it Schubert Polynomials for the classical groups}, J. Amer. Math. Soc., Volume 8, Number 2, April 1995
\bibitem{BGG73} I. N. Bernstein, I. M. Gel’fand, S. I. Gel’fand, {\it Schubert cells and cohomology of the spaces G/P}, Russian Math. Surveys 28, no. 3(1973) 1–26.

\bibitem{BSZ10} J. Bandlow, A. Schilling, M. Zabrocki, {\it The Murnaghan-Nakayama rule for k-Schur functions}, J. Comb. Theory. Ser. A, 118, no. 5(2011), 1588-1607.

\bibitem{BSS13} C. Berg, F. Saliola, L. Serrano, {\it The down operator and expansions of near rectangular k-Schur functions}, J. Comb. Theory. Ser. A 120(2013), no. 3, 623-636.

\bibitem{BSS14} C. Berg, F. Saliola, L. Serrano, {\it Pieri operators on the affine nilCoxeter algebra}, Trans. Amer. Math. Soc. 366(2014), 531-546.

\bibitem{BB05} A. Bj\"{o}rner, F. Brenti, {\it Combinatorics of Coxeter Groups}, Graduate Texts in Mathematics, Vol. 231.

\bibitem{CH08} L. Chen, M. Haiman, {\it A representation-theoretic model for k-atoms}, Talk 1039-05-169 at the AMS meeting in Claremont, McKenna, May 2008.

\bibitem{DM14} A. Dalal, J. Morse, {\it Quantum and affine Schubert calculus and Macdonald polynomials}, arXiv:1402.1464.
\bibitem{FK99} S. Fomin, A.N.Kirillov, {\it Quadratic algebras, Dunkl elements, and Schubert calculus}, Adv. Geom., Progress in Mathematics, 172(1999), 147-182.

\bibitem{FS94} S. Fomin, R. Stanley, {\it Schubert polynomials and the nil-Coxeter algebra}, Adv. Math. 103(1994), no. 2, 196–207.

\bibitem{Gra01} W. Graham, {\it Positivity in equivariant Schubert calculus}, Duke Math. J. 109(2001), no. 3, 599-614.

\bibitem{Hum90} J. E. Humphreys, {\it Reflection groups and Coxeter groups}, Cambridge Studies in Advanced Mathematics, 29. Cambridge University Press, Cambridge, 1990. xii+ 204 pp.

\bibitem{KS09} M. Kashiwara, M. Shimozono, {\it Equivariant K-theory of affine flag manifolds and affine Grothendieck polynomials}, Duke Math. J. 148 (2009), no. 3, 501-538.

\bibitem{Kir15} A. Kirillov, {\it On some quadratic algebras I$\frac{1}{2}$: Combinatorics of Dunkl and Gaudin elements, Schubert, Grothendieck, Fuss-Catalan, universal Tutte and Reduced polynomials}, preprint RIMS-1817, 2015.

\bibitem{KM03} A. Kirillov, T. Maeno, {\it Noncommutative algebras related with Schubert calculus on Coxeter groups}, European Journal of Combinatorics, 25 (2004), 1301-1325. Preprint RIMS-1437, 2003, 23p.

\bibitem{KM04} A. Kirillov, T. Maeno, {\it A note on quantization on Nichols algebra model for Schubert calculus on Weyl groups} Lett. in Math. Phys. 72 (2005) 233-241. Preprint RIMS-1481, 2004, 8p.

\bibitem{KM05} A. Kirillov, T. Maeno, {\it On some noncommutative algebras related with K-theory of flag varieties}, IRMN 60 (2005), 3753-3789, Preprint RIMS, 2005, 25p.
\bibitem{KM10} A. Kirillov, T. Maeno, {\it Nichols-Woronowicz model of coinvariant algebra of complex reflection groups}, J. Pure Appl. Algebra 214 (2010), 402-409 

\bibitem{KM12} A. Kirillov, T. Maeno, {\it Affine nil-Hecke algebras and braided differential structure on affine Weyl groups}, Publ. RIMS, no. 48 (2012), 215-228.

\bibitem{KK86}
B. Kostant, S. Kumar, {\it The nil Hecke ring and cohomology of $G/P$ for a Kac-Moody group $G$}, Adv. Math. 62 (1986), no. 3, 187-237.

\bibitem{Kum02} S. Kumar, {\it Kac-Moody groups, their flag varieties and representation theory}, Progress in Mathematics, 204. Birkh\''{a}user Boston, Inc., Boston, MA, 2002. xvi+606 pp.

\bibitem{Lam06} T. Lam, {\it Affine Stanley symmetric functions}, Amer. J. Math. 128(6):1553-1586, 2006.

\bibitem{Lam08} T. Lam, {\it Schubert polynomials for the affine Grassmannian}, J. Amer. Math. Soc. 21(2008), no.1, 259-281.

\bibitem{Lam11} T. Lam, {\it Affine Schubert classes, Schur positivity, and combinatorial Hopf algebras}, Bulletin of the LMS, 2011; doi:10.1112/blms/bdq110.
\bibitem{LS82} A. Lascoux, M. Sch\"{u}tzenberger, {\it Polyn\^{o}mes de Schubert}, C. R. Acad. Sci. Paris. S\'{e}r. I Math. 294 (13): 447-450 (1982).

\bibitem{LLM03} L. Lapointe, A. Lascoux, J. Morse, {\it Tableau atoms and a new Macdonald positivity conjecture}, Duke Math. J., Volume 116, Number 1 (2003), 103-146

\bibitem{LLMS10}
T. Lam, L. Lapointe, J. Morse, M. Shimozono, {\it Affine insertion and Pieri rules for the affine Grassmannian}, Mem. Amer. Math. Soc. 208 (2010), no. 977.

\bibitem{LLMS13} T. Lam, L. Lapointe, J. Morse, M. Shimozono, {\it k-shape poset and branching of k-Schur functions}, Mem. Amer. Math. Soc., 223 (2013), no. 1050.
\bibitem{LLMSSZ}
T. Lam, L. Lapointe, J. Morse, A. Schilling, M. Shimozono, M. Zabrocki, {\it k-Schur functions and affine Schubert calculus}, Fields Institute Monographs 33 (2014), VIII. 219, p.126 illus.

\bibitem{LLS15} T. Lam, S. Lee, M. Shimozono, in preperation.

\bibitem{Lee14} S. Lee, {\it Pieri rule for the affine flag variety}, Adv. Math. 304 (2017), 266-284.

\bibitem{LM05} L. Lapointe, J. Morse, {\it Tableaux on $k+1$-cores, reduced words for affine permutations, and k-Schur expansions}, J. Comb. Theory Ser. A 112 (2005), no. 1, 44–81.

\bibitem{MPP14} K. M\'{e}sz\'{a}ros, G. Panova, A. Postnikov, {\it Schur times Schubert via the Fomin-Kirillov algebra}, Electr. J. Comb. Vol 21, Issue 1 (2014)
\bibitem{MS15} J. Morse, A. Schilling, {\it Crystal Approach to Affine Schubert Calculus}, Int. Math. Res. Notices., doi:10.1093/imrn/rnv194.

\bibitem{Pet97} D. Peterson, Lecture Notes at MIT, 1997.
\bibitem{Sta99} R. Stanley, {\it Enumerative Combinatorics}, Cambridge University Press, Vol. 2(1999), ISBN 0-521-56069-1.
\bibitem{Qui} D. Quillen, unpublished.
\end{thebibliography}
\end{document}